\documentclass{amsart}

\usepackage{rotating}
\usepackage{amsmath, amsthm, amssymb, fullpage, xypic}
\usepackage[all, knot]{xy}
\usepackage{multirow}
\xyoption{poly}

\newtheorem{theorem}{Theorem}
\newtheorem{lemma}[theorem]{Lemma}
\newtheorem{proposition}[theorem]{Proposition}
\newtheorem{definition}[theorem]{Definition} 
\newtheorem{corollary}[theorem]{Corollary}

\newcommand{\Mdef}[2]{\newcommand{#1}{\relax\ifmmode #2 \else $#2$\fi}}

\Mdef{\cV}{\mathcal{V}}
\Mdef{\cF}{\mathcal{F}}
\Mdef{\cE}{\mathcal{E}}
\Mdef{\cP}{\mathcal{P}}
\Mdef{\cS}{\mathcal{S}}
\newcommand{\R}{\mathbb{R}}
\newcommand{\Z}{\mathbb{Z}}

\DeclareMathOperator{\Cok}{Cok}
\DeclareMathOperator{\Ker}{Ker}

\newcommand{\edit}{}        

\newcommand{\lra}{\longrightarrow}

\newcommand{\bv}{\left[\begin{array}{r}}
\newcommand{\brr}{\left[\begin{array}{rr}}
\newcommand{\ev}{\end{array}\right]}

\begin{document}
\title{Generalized {\em Tonnetze}}
\date{July 10, 2011}
\author[M.J.~Catanzaro]{Michael J.\ Catanzaro}
\address{Department of Mathematics, Wayne State University, Detroit, Michigan  48202}
\email{mike@math.wayne.edu}
\begin{abstract}
We study a generalization of the classical Riemannian {\it Tonnetz}
to $N$-tone equally tempered scales (for all $N$) and arbitrary
triads.
We classify all the spaces which result.
The torus turns out to be the most common possibility,
especially as $N$ grows.  Other spaces include 2-simplices, tetrahedra
boundaries, and the harmonic strip
(in both its cylinder and M\"obius band variants).
The final and most exotic space we find is something we call
a `circle of tetrahedra boundaries'.  These are the {\em Tonnetze}
for spaces of triads which contain a tritone. They are closely
related to Peck's Klein bottle {\em Tonnetz}.
\end{abstract}
\maketitle
{\bf Keywords:} {\em Tonnetz}, $N$-tone equal temperament, space of triads, torus, simpicial complex, geometry of chords, generalized scales

\section{Introduction}

Geometric aspects of chord spaces have been considered by many authors, for instance \cite{Hyer,Tymoczko,Cohn,Cohn2,Douthett,Gollin,
Mazzola,Peck}.
We continue this tradition by enumerating all possible geometries
that arise as {\it generalized Tonnetze} in the following way. To any fixed, unordered, three-note chord $\{a,b,c\}$ in
$\mathbb{Z}/N$ one may associate a 2-dimensional simplicial complex. Its 2-simplices are the mod $N$ transpositions and
inversions of $\{a,b,c\}$, its 1-simplices are two-element subsets of transpositions and inversions of $\{a,b,c\}$, and a
0-simplex is an element of $\mathbb{Z}/N$. In Theorem~\ref{sumthm},  we prove that these simplicial 
complexes comprise the following list: 2-simplices, tetrahedra boundaries, tori, cylinders, M\"obius bands,
and circles of tetrahedra boundaries.

A prime example of such a simplicial complex is the neo-Riemannian {\it Tonnetz}, which is obtained by taking $\{a,b,c\}$ to
be any major or minor triad. The 0-simplices are the elements of $\mathbb{Z}/{12}$, while the 1-simplices are major thirds
$\{x,x+4\}$, minor thirds $\{y,y+3\}$ and perfect fifths $\{z,z+7\}$ for $x,y,z \in \mathbb{Z}/{12}$. The 2-simplices are the
major and minor triads. We call this simplicial complex $C(3,4,5)$ because the step intervals of the minor chord are 3, 4, and
5. See Figure~\ref{torus} for a representation of this torus in the plane (opposite sides are identified to make the torus).  We
emphasize that the construction of such simplicial complexes is direct, {\it it does not rely on any planar representation as
an intermediate step}.

Indeed, there may not even be a canonical planar representation. For example, if we begin with $\{0,1,3\}$ in the mod 6
universe $\mathbb{Z}/6$, then the resulting simplicial complex consists of boundaries of three tetrahedra boundaries
\begin{eqnarray*}
\{0134\} & = & \{\{013\}, \{014\}, \{034\}, \{134\} \} \\
\{1245\} & = & \{\{124\}, \{125\}, \{145\}, \{245\} \} \\
\{2350\} & = & \{\{235\}, \{230\}, \{250\}, \{350\} \} 
\end{eqnarray*}
glued together along 1-simplices of the form $\{x,x+3\}$. We call this simplicial complex $C(1,2,3)$ because the step
intervals of $\{0,1,3\}$ in $\mathbb{Z}/6$, from smallest to largest, are 1, 2, and 3.

It turns out that $C(1,2,3)$ is closely related to Peck's Klein bottle {\em Tonnetz} \cite{Peck} in a manner
explained in detail in Section~\ref{exsptr}. Briefly, our space $C(1,2,3)$ is not a surface
because the tritone intervals are shared by four 2-simplices. In Peck's Klein bottle {\em Tonnetz},
each of these tritones is represented by two 1-simplices, each of which is shared by two 2-simplices,
thereby forming a surface, the Klein bottle. As we show in Section~\ref{exsptr}, there are choices
involved in the process of splitting the tritones in two, which lead to 8 versions of Peck's Klein
bottle {\em Tonnetz}, half of which are Klein bottles, and half of which are tori.

The simplicial complexes we consider may not even be connected. For example, the augmented triads form the 2-simplices of
$C(4,4,4)$, since they 
are characterized by the fact that there are \edit{four} semi-tones between
each note in the chord.
$C(4,4,4)$ has four components, $\{C,E,G^{\sharp}\}$, $\{C^{\sharp},F,A\}$, 
$\{D,F^{\sharp}, A^{\sharp}\}$,
and $\{D^{\sharp}, G, B\}$, each of which consists of a single \edit{2-simplex} (i.e., chord). 
The complete symmetry of the augmented triad
means that we have only four of them in the chromatic scale, corresponding
to the four disjoint \edit{2-simplices} of $C(4,4,4)$. 

Another example of a non-connected generalized {\it Tonnetz} is $C(3,3,6)$, whose 2-simplices are the diminished triads,
since the \edit{diminished triad} consists of three notes which
differ by three, three, and six semi-tones. The space they form, $C(3,3,6)$, can be decomposed into
three disjoint copies of $C(1,1,2)$, which is a tetrahedron. 
Each tetrahedron contains four 
of the twelve possible diminished \edit{triads}. For example, one of the three
contains the 4 chords \edit{$\{C,D^{\sharp},F^{\sharp}\}$, $\{F^{\sharp},A,C\}$,
$\{A,F^{\sharp}, D^{\sharp}\}$, $\{A,C,D^{\sharp}\}$}.
The remaining diminished \edit{triads} fall into two other
families like this, and these families share no notes.

After these illustrations, we may now state our main theorems precisely.
Let $n_1$, $n_2$, $n_3$, and $N$ be positive integers such that $n_1+n_2+n_3=N$ and $1 \leq n_1 \leq n_2 \leq n_3 < N$. Then
the connected {\bf components} of the generalized {\it Tonnetz} $C(n_1,n_2,n_3)$ are 2-simplices, tetrahedra boundaries, tori,
cylinders, M\"obius bands, or circles of tetrahedra boundaries (Theorem~\ref{sumthm}). 
Moreover, any two connected components of $C(n_1,n_2,n_3)$ are isomorphic (Theorem~\ref{disj}). The
simplicial complex $C(n_1,n_2,n_3)$ is connected if and only if $gcd(n_1,n_2,n_3)=1$ (Theorem~\ref{connected}). 
The complete list of possibilities is contained in Figure~\ref{possibilities}, and is proved in Sections~\ref{surfwobdy},
\ref{surfwbdy}, and \ref{nonsurf}.

\begin{figure}
\renewcommand{\arraystretch}{1.2}
\begin{center}
\begin{tabular}{|c|c|}
  \hline
  Connected component of $C(n_1,n_2,n_3)$ & Relations on $n_1,n_2,$ and $n_3$ \\
   \hline
   2-simplex & $n_1 = n_2 = n_3$ \\
   \hline
   tetrahedron boundary& $n_1 = n_2 < n_3 = N/2$ \\
   \hline
   \multirow{2}{*}{cylinder} & $n_1 = n_2 < n_3 \neq N/2$ or \\
            & $n_1 < n_2 = n_3$ with $N$ even \\
   \hline
   \multirow{2}{*}{M\"obius band} &  $n_1 = n_2 < n_3 \neq N/2$ or \\
            & $n_1 < n_2 = n_3$ with $N$ odd \\
   \hline
   circle of $(n_1 +n_2)/\gcd(n_1,n_2)$& \multirow{2}{*}{$n_1 + n_2 = n_3 = N/2$} \\
   tetrahedra boundaries & \\
   \hline
   torus  & $n_1 < n_2 < n_3$ \\
   \hline
\end{tabular}
\end{center}
\caption{The classification of generalized {\em Tonnetze}.}
\label{possibilities}
\end{figure}

In Figure~\ref{12cases}, we list all the possibilities for $\mathbb{Z}/12$.  
Note that \edit{$\mathbb{Z}/{12}$} contains every possibility except the
M\"obius band.  This is the best we could expect, since the M\"obius band
occurs only in scales with an odd number of notes.  (See Theorem
\ref{nnnplusk}.)


\edit{The paper is organized as follows.
We begin with definitions and examples in Section~\ref{predef}.
Next, in Section~\ref{eulchar}, we count the number of vertices, edges, and
faces in $C(n_1,n_2,n_3)$. Using this information 
we determine the Euler characteristic of $C(n_1,n_2,n_3)$.
We use this to classify those spaces of triads which
are compact, connected, orientable surfaces in Section~\ref{classpa}. We discuss the connectedness of  
spaces of triads in Section~\ref{topprop}. If the space is not connected, it must consist of
disjoint, identical subspaces. This, along with  
homogeneity (using the operations of the $T/I$ group), allows us 
to classify the remaining spaces. 
We use ideas from geometry, algebraic topology, homological algebra, and 
elementary number theory.
The method of proof is explicit: we consider a
general vertex (pitch class) $a$ and all 2-simplices (triads) which contain it.
The discussion of connected spaces of triads naturally falls into
three cases: 
surfaces with boundary, surfaces without boundary, and spaces which are 
not surfaces.}

\begin{figure}
\renewcommand{\arraystretch}{1.2}
\begin{tabular}{|c|c|c|}
  \hline
  $C(n_1,n_2,n_3)$ & Simplicial Complex & Representative Trichord \\
  \hline
  $C(1,1,10)$ & cylinder & $\{ 0,1,2 \} = \{C,C^{\sharp},D\}$ \\
  $C(1,2,9)$ & torus & $\{ 0,1,3 \} = \{C,C^{\sharp},D^{\sharp}\} $\\
  $C(1,3,8)$ & torus & $\{ 0,1,4 \} = \{C,C^{\sharp},E\}$ \\
  $C(1,4,7)$ & torus & $\{ 0,1,5 \} = \{C,C^{\sharp},F\}$ \\
  $C(1,5,6)$ & circle of 6 tetrahedra boundaries & $ \{ 0,1,6 \} = \{C,C^{\sharp},F^{\sharp}\}$ \\
  $C(2,2,8)$ & two disjoint cylinders & $\{ 0,2,4 \} = \{C,D,E\}$  \\
  $C(2,3,7)$ & torus & $\{ 0,2,5 \} = \{C,D,F\}$ \\
  $C(2,4,6)$ & two disjoint circles of 3 tetrahedra boundaries & $\{ 0,2,6 \} = \{C,D,F^{\sharp}\}$ \\
  $C(2,5,5)$ & cylinder & $\{ 0,2,7 \} = \{C,D,G\}$  \\
  $C(3,3,6)$ & three disjoint tetrahedra boundaries & $\{ 0,3,6 \} = \{C,E^{\flat},G^{\flat}\}$ \\
  $C(3,4,5)$ & torus & $\{0,3,7 \} = \{C,E^{\flat},G\} $\\
  $C(4,4,4)$ & four disjoint \edit{2-simplices} & $\{ 0,4,8\} = \{C,E,G^{\sharp}\}$ \\
  \hline
\end{tabular}
\caption{The generalized {\em Tonnetze} found in $\Z/12$.}
\label{12cases}
\end{figure}


We can see that the most like case is a torus, by the following argument.
\edit{
The number of possible shapes $\{n_1,n_2,n_3\}$ in an $N$-note scale is
$aN^2 +$ lower degree terms. (In fact, $a= 1/12$ but this is not important). This follows
since the number of choices for each of $n_1$ and $n_2$ is proportional to $N$, and $n_3$
is then determined by $n_3 = N -n_1 - n_2$. Each of the non-torus cases is
determined by one or more equalities, such as $n_1 = n_2$. The number of 
triples $\{n_1,n_2,n_3\}$ satisfying this in addition is $bN + $constant, for some $b$.
Adding all the non-torus cases, we find their proportion is
\[
   \frac{cN + \mbox{ constant}}{aN^2 + \mbox{ linear } + \mbox{ constant}}
\]
which goes to 0 as $N \rightarrow \infty$.
Hence, as $N$ grows, the probability of \emph{not} finding a torus
decreases to zero.}


In Section~\ref{exsptr}, we compare our results to those of Balzano, Cohn, Mazzola,
and Peck. In the first three cases, our results are the same. Peck's Klein
bottle {\em Tonnetze}, on the other hand, can be mapped to our spaces, but are not identical.
 
\section{Preliminary Definitions}
\label{predef}

In order to help visualize these complexes, we arrange the $N$ note
scale symmetrically on a circle. We number the notes in the scale from 0
to $N-1$ in an increasing, clockwise fashion on the perimeter of the
circle. In Figure~\ref{345}, we see the musical circle for $\Z/12$
with a triad inscribed in it.

\begin{figure}
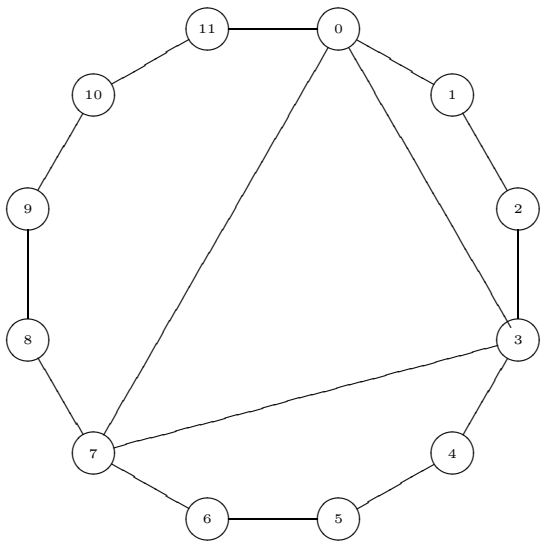

\[
\def\objectstyle{\scriptscriptstyle}
\xy
\xygraph{!{/r8pc/:}!P12"B"{\cir<8pt>{}}"B1"2"B2"1"B3"0"B4"{11}"B5"{10}"B6"9"B7"8"B8"7"B9"6"B10"5"B11"4"B12"3}
\PATH~={**\dir{-}}'"B3"'"B8"'"B12"'
\endxy
\]
\caption{The C-minor triad $\{0,3,7\}$ is a 2-simplex of $C(3,4,5)$.}
\label{345}
\end{figure}

Let us show $C(3,4,5)$ is the neo-Riemmanian {\it Tonnetz}
\edit{of major and minor triads in more detail.  This is} 
discussed in many sources including \cite{Riemann}, \cite{Cohn2}.
 Throughout this section, we \edit{will} be referring to the
 \edit{musical} circle in Figure~\ref{345}, \edit{with} the
 $\{0,3,7\}$ triad inscribed in it. Let us label \edit{pitch classes} in
 the \edit{usual way}:
\begin{center}
\begin{tabular}{|c|c|c|c|c|c|c|c|c|c|c|c|}
\hline
 0&1&2&3&4&5&6&7&8&9&10&11 \\
\hline 
 C&C$^{\sharp}$&D&D$^{\sharp}$&E&F&F$^{\sharp}$&G&G$^{\sharp}$&A&A$^{\sharp}$&B\\
 \hline
\end{tabular}
\end{center}
We see that the triad we have been referring to, \edit{$\{0,3,7\}$}, is the c-minor triad. 
\edit{Other common examples are $C = \{0,4,7\}$, $E = \{4,8,11\}$, and $F= \{0,5,9\}$.}
If we construct these \edit{2-simplices} from our musical circle and
\edit{attach} them as in Figure~\ref{glue}, we see the space in
Figure~\ref{torus}.
Let $1 \le n_1 \le n_2 \le n_3$ and $N = n_1 + n_2 + n_3$. A type of triad in an
$N$-note scale can be described by giving the \edit{distances} $n_1$,
$n_2$, and $n_3$ between the notes arranged \edit{cyclically}. Thus, a minor
triad \edit{in $\Z/{12}$} corresponds to $(n_1, n_2, n_3) = (3,4,5)$
since, for example, there are 3 half steps from C to E$^{\flat}$, 4 half
steps from E$^{\flat}$ to G, and 5 from G to C, always counting upward in
pitch. \edit{Since the number of notes in our scale is $N$, we shall use
the group $\Z/N$ of integers modulo $N$
to label the notes}.

\edit{The mathematical object we use to model a space of triads is that
of a simplicial complex, which is a set theoretic abstraction of
polyhedral geometry. 
Formally, a simplicial complex consists of a finite set $\cV$
of \emph{vertices} together with a set $\cS$ of subsets of $\cV$ 
called \emph{simplices} with the property that any subset of a simplex
is again a simplex.  We call subsets in $\cS$ with $n+1$ elements \emph{$n$-simplices}.
}

\edit{
Each simplicial complex has a geometric realization as a polyhedron, obtained by
associating to
each 0-simplex $\{v\}$ a point which we shall call $v$, 
to each 1-simplex $\{v, w\}$ an interval from $v$ to $w$,
to each 2-simplex $\{u,v,w\}$ a solid triangle bounded by the 1-simplices $\{u,v\}$,
$\{u,w\}$ and $\{v,w\}$ (so that its vertices are $u$, $v$ and $w$) and so on.
}

\edit{
The condition that subsets of simplices are themselves simplices implies that
as we build up the complex, the lower dimensional simplices we need as boundaries
of higher dimensional ones are already present.
}

\edit{For example, the simplicial complex with
$\cV = \{0,1,2\}$ and $\mathcal{S} = \{ \{0\}, \{1\}, \{2\} \}$
represents the vertices of a triangle. However, the complex
$\cV = \{ 0, 1, 2 \}$, $\mathcal{S} = \{ \{0, 1\}, \{1, 2\}, \{0\},
\{1\}, \{2\} \}$ represents the vertices as well as two edges of a
triangle. The geometric realization of these two examples can be seen in
Figure~\ref{simpcompex}.}

\begin{figure}
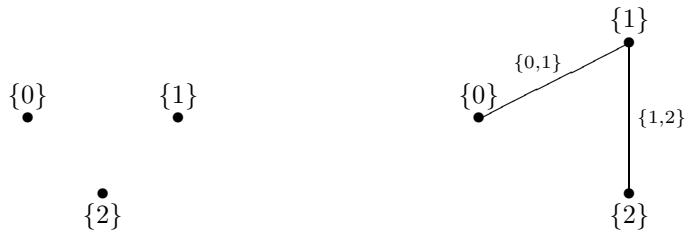

 \[
\xy 
(10,0)*{\bullet};(10,-3)*{\{2\}};(0,10)*{\bullet};(0,13)*{\{0\}};
(20,10)*{\bullet};(20,13)*{\{1\}};
(60,10)*{\bullet};(60,13)*{\{0\}};(80,0)*{\bullet};
(80,-3)*{\{2\}};(80,23)*{\{1\}};(80,20)*{\bullet}; 
\ar@{-} (60,10);(80,20)^{\{0,1\}}
\ar@{-} (80,20);(80,0)^{\{1,2\}}

\endxy
\]
\caption{Two examples of simplicial complexes}
\label{simpcompex}
\end{figure}

\edit{A simplicial complex consisting of all non-empty subsets of an
$(n+1)$ element set is called an \emph{$n-$simplex} because its
geometric realization is the simplest n-dimensional polyhedron. 
Any simplicial complex is the union of
the simplices contained within it.}

\edit{One advantage of simplicial complexes is that they allow one to study polyhedral topology by working with
finite sets. In particular, \emph{simplicial homology} is extraordinarily useful. 
We shall use it in Section~\ref{connectedness} to determine the number of connected components in 
our \emph{Tonnetze}, so we shall not be concerned with the general definition here. However, we would like to
suggest one idea. Intuitively, the $n-$th homology measures the `n-dimensional holes' in a space, so it
may be a useful tool in analyses. For example, the neo-Riemannian \emph{Tonnetz} forms a torus, so that there are two
independent paths on it, out of which all others can be composed. This suggests that there are essentially
two independent ways in which one can form chord progressions which start and end at the same chord
without repeating themselves. Naturally, this is mere speculation, and would require careful analysis to determine 
whether homology groups have musical significance (see \cite{Netske} for
more on this topic).}

We label the distances between notes in a least to greatest fashion, so
$1 \leq n_1 \leq n_2 \leq n_3$. If we denote the number of notes in the
scale by $N$, then $N=n_1 + n_2 + n_3$ since this corresponds to a complete
circuit of the musical circle for $C(n_1,n_2,n_3)$.

\begin{definition} \edit{Let $N$ be a positive integer and $n_1$, $n_2$, $n_3$ be positive integers such that $n_1 + n_2 + n_3 = N$ and $1 \le n_1 \le n_2 \leq n_3 < N$. We denote by $C(n_1,n_2,n_3)$ the abstract simplicial complex in which
  \begin{enumerate}
   \item the set of 0-simplices is $\mathbb{Z}/N$,
   \item the set of 1-simplices consists of all mod $N$ translations of $\{0,n_1\},\{0,n_2\},\{0,n_3\}$, and
   \item the set of 2-simplices consists of all mod $N$ translations and inversions of $\{0,n_1,n_1+n_2\}$.
   \end{enumerate}
In other words, the set of 2-simplices of $C(n_1,n_2,n_3)$ is the $T/I$-class of $\{0,n_1,n_1+n_2\}$ and the 0- and 1-simplices are the vertices and edges of these 2-simplices. The geometric realization of this abstract simplicial complex, which is also denoted $C(n_1,n_2,n_3)$, is then a simplicial complex called a \emph{space of triads}.}
\end{definition}

Let $\cF$ be the set of all faces, $\cE$ be the set of all edges, and $\cV$ be the set of all vertices of $C(n_1,n_2,n_3)$. To view $C(n_1,n_2,n_3)$ as a space
via geometric realization, take a solid triangle for each 2-simplex and join them along common edges. For example, in $C(3,4,5)$, the face $\{0,3,7\}$ and the face $\{3,7,10\}$ would be joined along the edge $\{3,7\}$ (see Figure~\ref{glue}). \\

\begin{figure}

\[
\xymatrix{
&
3
\ar@{-}[ddl]
\ar@{-}[ddr]
\ar@{<.>}[rr]
&&
3
\ar@{-}[rr]
\ar@{-}[ddr]
&&
10
\ar@{-}[ddl]
&&
3
\ar@{-}[ddl]
\ar@{-}[ddr]
\ar@{-}[rr]
&&
10
\ar@{-}[ddl]
\\
&&
\ar@{<.>}[r]
&&&
\ar@{=>}[r]
&&&&
\\
0
\ar@{-}[rr]
&&
7
\ar@{<.>}[rr]
&&
7
&&
0
\ar@{-}[rr]
&&
7
&&
\\
}
\]

\caption{Gluing the edges together.}
\label{glue}
\end{figure}
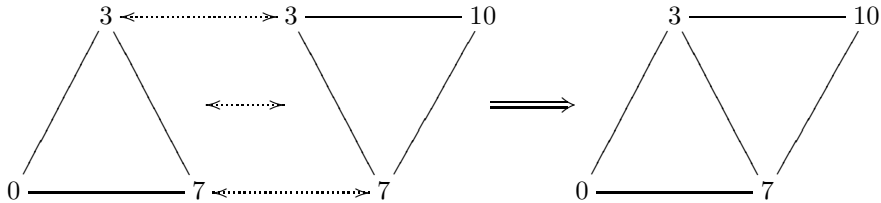

\section{Examples of Spaces of Triads}
\label{exsptr}

Let us show $C(3,4,5)$ is the neo-Riemmanian {\em Tonnetz} of
major and minor triads in more detail. In Figure~\ref{torus},
we have shown each \edit{2-simplex} (face) exactly once. Edges along the
boundary of our drawing occur twice and so are joined in the space
$C(3,4,5)$. The top and bottom vertices begin to repeat, and we see that
\edit{these edges} are actually equal. \edit{Similarly,} the left and right
sides are equal. Thus, we can take the top edge and glue it on to
the bottom. Similarly gluing the left and right edges, we see that we
have a torus. Hence, the space $C(3,4,5)$, of major and minor triads
 in \edit{$\Z/{12}$}, is a torus.
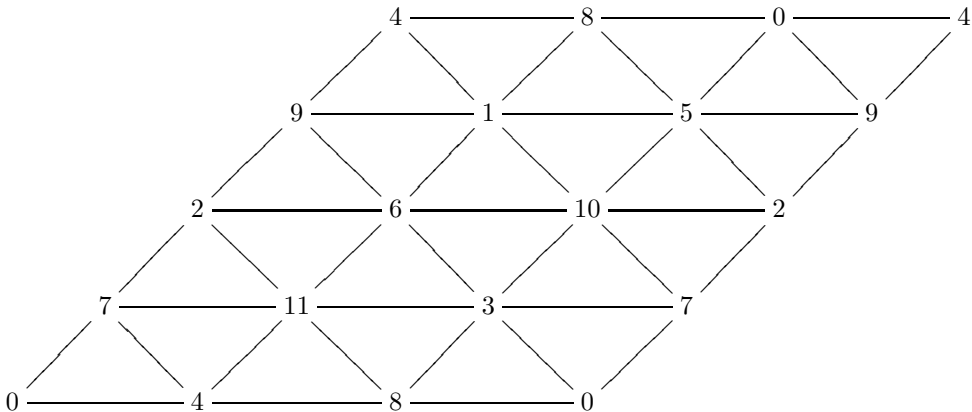
\begin{figure}
\[
\xymatrix{
&&&
&
4
\ar@{-}[dl]
\ar@{-}[dr]
\ar@{-}[rr]
&&
8
\ar@{-}[dl]
\ar@{-}[dr]
\ar@{-}[rr]
&&
0
\ar@{-}[dl]
\ar@{-}[rr]
\ar@{-}[dr]
&&
4
\ar@{-}[dl]
\\
&
&&
9
\ar@{-}[dl]
\ar@{-}[rr]
\ar@{-}[dr]
&&
1
\ar@{-}[dl]
\ar@{-}[rr]
\ar@{-}[dr]
&&
5
\ar@{-}[dl]
\ar@{-}[rr]
\ar@{-}[dr]
&&
9
\ar@{-}[dl]
\\
&&
2
\ar@{-}[dl]
\ar@{-}[rr]
\ar@{-}[dr]
&&
6
\ar@{-}[dl]
\ar@{-}[rr]
\ar@{-}[dr]
&&
10
\ar@{-}[dr]
\ar@{-}[dl]
\ar@{-}[rr]
&&
2
\ar@{-}[dl]
&&&
\\
&
7
\ar@{-}[dl]
\ar@{-}[rr]
\ar@{-}[dr]
&&
11
\ar@{-}[dl]
\ar@{-}[rr]
\ar@{-}[dr]
&&
3
\ar@{-}[dl]
\ar@{-}[rr]
\ar@{-}[dr]
&&
7
\ar@{-}[dl]
&&&&
\\
0
\ar@{-}[rr]
&&
4
\ar@{-}[rr]
&&
8
\ar@{-}[rr]
&&
0
&&&&&
\\
}
\]
\caption{A fundamental region for $C(3,4,5)$.}
\label{torus}
\end{figure}

\edit{Mazzola's harmonic strip \cite[pp. 321-322]{Mazzola} is
another interesting example.
The harmonic strip has vertices $\Z/7$. Specifically, it has seven 
0-simplices, fourteen 1-simplices, and seven 2-simplices. 
 By Theorem~\ref{euler}, this space
must either be $C(1,1,5)$, $C(1,3,3)$ or $C(2,2,3)$. 
Theorem~\ref{nnnplusk} implies this space will be a M\"obius band,
in agreement with Mazzola's findings.}

In \cite{Balzano}, Balzano discusses 12-fold and microtonal pitch systems. 
He describes several examples, including different spaces of triads with 20, 30, and
42 notes \cite[pp. 77-80]{Balzano}. These spaces are $C(4,5,11)$,
$C(5,6,19)$, and $C(6,7,29)$ respectively.  He also discusses the analogs
of the Riemannian {\em Tonnetz} generated by thirds and fifths from the diatonic scale
in the more general $C(n,n+1,n^2-n-1)$.
\cite[p. 75]{Balzano}. Theorem~\ref{tori} implies that this space will
always be a torus.

\edit{Cohn's work in \cite{Cohn} concerning the realizations of \emph{Parsimonious 
Tonnetze} also requires mention. In \cite[Sec. 2.3]{Cohn}, Cohn computes the 
requirements on $N,$ $n_1,$ and $n_2$ such that they lead to `optimally parsimonious
voice-leading' under PLR-family operations. 
He proves that the only spaces which support parsimonious voice-leading are
the $C(n,n+1,n+2)$. By Theorem~\ref{tori}, these \emph{Parsimonious Tonnetze} are
tori when $n > 1$ (when $n=1$, a more exotic space occurs).}

\edit{Cohn also describes the $LPR$ loops of the Riemannian \emph{Tonnetz}
in \cite{Cohn}. An $LPR$ loop consists of a triad and all other triads obtained by
applying the $R$, $P$, and $L$ transformations to it, in that order, until returning
to the original triad. This is of interest in connection with parsimonious voice-leading.
In \cite[Sec. 3.6]{Cohn}, Cohn discusses the placement of the six triads in an $LPR$ loop around a
given vertex, as shown in our Figure~\ref{hexagon}. The \edit{cyclical} ordering of the triads containing a particular
pitch class by means of an $LPR$ loop play a significant role in determining which generalized
\emph{Tonnetze} form surfaces.
This is not possible in every case: specifically when $C(n_1,n_2,n_3)$ is a cylinder,
M\"obius band, or circle of tetrahedra boundaries. 
In Section~\ref{surfwbdy} (notably Figures~\ref{band} and \ref{band2}),
we discuss the cases in which
a full LPR loop cannot be made around every vertex. In these cases, though,
the $P$, $L$, and $R$ operations do not behave as they do classically. For example,
if $n_1 = n_2$, then $P$ becomes the identity map.}


One space which does not occur as one of our generalized {\em Tonnetze} is
the Klein bottle.
Peck, on the other hand, discovers Klein bottle {\em Tonnetze} in his analysis
of musical spaces \cite{Peck}. Peck's Klein
bottle {\em Tonnetze} occur
in those cases in which we find circles of
tetrahedra boundaries.  Theorem~\ref{circleoftetra} implies that these are exactly the cases in which the
chords contain tritones (intervals of length $N/2$).
In these cases, our {\em Tonnetze} are quotients of his,
obtained by collapsing certain pairs of edges.  In Peck's Klein bottle {\em Tonnetze},
there are two edges corresponding to each tritone, while in ours,
these are considered to be a single edge.  This allows Peck's
{\em Tonnetze} to be  surfaces, while ours are singular along these edges,
since four distinct 2-simplices are joined at them.

Peck \cite[Fig. 20]{Peck} shows the pitch class {\em Tonnetz} consisting of the even pitch classes in
C(2,4,6).  We reproduce his figure here as Figure~\ref{peck1}.
Examining it, the relation between the rows is key.
Moving from one row to the next is done by applying an
operation on triads which fixes the tritone and moves the
third note.  In {\em Tonnetze} of chords which do not contain a tritone, 
there is a unique operation in the $PLR$ group which does this.
However, when the fixed interval is a tritone, there are
four triads which share it, so there are three other triads to which the triad could
move.  
In these {\em Tonnetze}, triads fall into two types.  The transpositions of $\{0,n_1,N/2\}$
are distinct from their inversions, which are transpositions of $\{0,n_2,N/2\}$.
(Recall that $n_1+n_2 = N/2$ in the {\em Tonnetze} under consideration.)

\begin{figure}
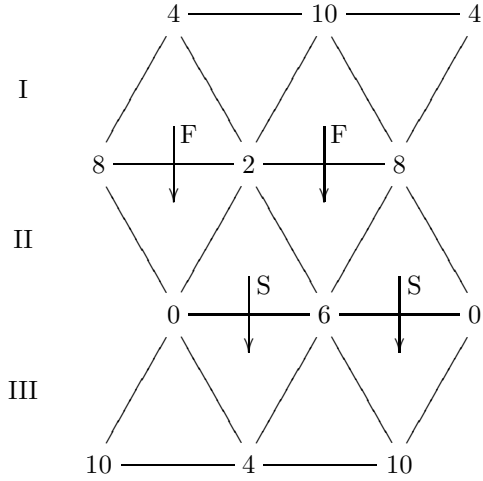

\begin{center}
\begin{minipage}{0.3\textwidth}
\vspace{0pt}
$
\xy
(-10,50)*{\mathrm{I}};
(-10,30)*{\mathrm{II}};
(-10,10)*{\mathrm{III}};
(0,0)*{10};(20,0)*{4};(40,0)*{10};(10,20)*{0};(30,20)*{6};
(50,20)*{0};(0,40)*{8};(20,40)*{2};(40,40)*{8};(10,60)*{4};
(30,60)*{10};(50,60)*{4};(12,44)*{\mathrm{F}};(32,44)*{\mathrm{F}};
(22,24)*{\mathrm{S}};(42,24)*{\mathrm{S}};
\ar@{-}(3,0);(18,0)
\ar@{-}(22,0);(37,0)
\ar@{-}(12,20);(28,20)
\ar@{-}(32,20);(48,20)
\ar@{-}(2,40);(18,40)
\ar@{-}(22,40);(38,40)
\ar@{-}(12,60);(27,60)
\ar@{-}(33,60);(48,60)
\ar@{-}(1,3);(9,17)
\ar@{-}(11,17);(19,3)
\ar@{-}(21,3);(29,17)
\ar@{-}(31,17);(39,3)
\ar@{-}(41,3);(49,17)
\ar@{-}(1,37);(9,23)
\ar@{-}(11,23);(19,37)
\ar@{-}(21,37);(29,23)
\ar@{-}(31,23);(39,37)
\ar@{-}(41,37);(49,23)
\ar@{-}(1,43);(9,57)
\ar@{-}(11,57);(19,43)
\ar@{-}(21,43);(29,57)
\ar@{-}(31,57);(39,43)
\ar@{-}(41,43);(49,57)
\ar@{->}(10,45);(10,35)
\ar@{->}(30,45);(30,35)
\ar@{->}(20,25);(20,15)
\ar@{->}(40,25);(40,15)
\endxy
$
\end{minipage}
\vspace{4ex}
\end{center}
\caption{Peck's Klein bottle {\em Tonnetz} \cite[Fig. 20]{Peck}
\label{peck1}}
\end{figure}

\begin{figure}
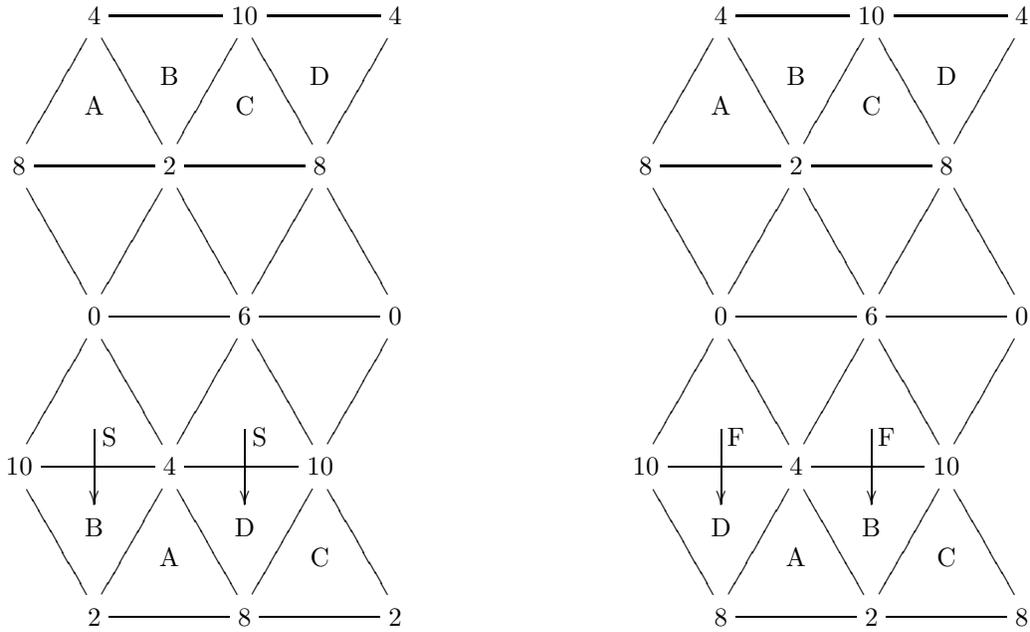

\begin{center}
\begin{minipage}{\textwidth}
\scalebox{1}{
\begin{minipage}{0.3\textwidth}
\vspace{0pt}
$
\xy
(0,0)*{10};(20,0)*{4};(40,0)*{10};(10,20)*{0};(30,20)*{6};
(50,20)*{0};(0,40)*{8};(20,40)*{2};(40,40)*{8};(10,60)*{4};
(30,60)*{10};(50,60)*{4};(10,-20)*{2};(30,-20)*{8};(50,-20)*{2};
(10,48)*{\mathrm{A}};(20,52)*{\mathrm{B}};(30,48)*{\mathrm{C}};
(40,52)*{\mathrm{D}};(10,-8)*{\mathrm{B}};(20,-12)*{\mathrm{A}};
(30,-8)*{\mathrm{D}};(40,-12)*{\mathrm{C}};(12,4)*{\mathrm{S}};
(32,4)*{\mathrm{S}};
\ar@{-}(3,0);(18,0)
\ar@{-}(22,0);(37,0)
\ar@{-}(12,20);(28,20)
\ar@{-}(32,20);(48,20)
\ar@{-}(2,40);(18,40)
\ar@{-}(22,40);(38,40)
\ar@{-}(12,60);(27,60)
\ar@{-}(33,60);(48,60)
\ar@{-}(32,-20);(48,-20)
\ar@{-}(12,-20);(28,-20)
\ar@{-}(1,3);(9,17)
\ar@{-}(11,17);(19,3)
\ar@{-}(21,3);(29,17)
\ar@{-}(31,17);(39,3)
\ar@{-}(41,3);(49,17)
\ar@{-}(1,37);(9,23)
\ar@{-}(11,23);(19,37)
\ar@{-}(21,37);(29,23)
\ar@{-}(31,23);(39,37)
\ar@{-}(41,37);(49,23)
\ar@{-}(1,43);(9,57)
\ar@{-}(11,57);(19,43)
\ar@{-}(21,43);(29,57)
\ar@{-}(31,57);(39,43)
\ar@{-}(41,43);(49,57)
\ar@{-}(1,-3);(9,-17)
\ar@{-}(11,-17);(19,-3)
\ar@{-}(21,-3);(29,-17)
\ar@{-}(31,-17);(39,-3)
\ar@{-}(41,-3);(49,-17)
\ar@{->}(10,5);(10,-5)
\ar@{->}(30,5);(30,-5)
\endxy
$
\end{minipage}
}
\hspace{20ex}
\scalebox{1}{
\begin{minipage}{0.3\textwidth}
\vspace{0pt}
$
\xy
(0,0)*{10};(20,0)*{4};(40,0)*{10};(10,20)*{0};(30,20)*{6};
(50,20)*{0};(0,40)*{8};(20,40)*{2};(40,40)*{8};(10,60)*{4};
(30,60)*{10};(50,60)*{4};(10,-20)*{8};(30,-20)*{2};(50,-20)*{8};
(10,48)*{\mathrm{A}};(20,52)*{\mathrm{B}};(30,48)*{\mathrm{C}};
(40,52)*{\mathrm{D}};(10,-8)*{\mathrm{D}};(20,-12)*{\mathrm{A}};
(30,-8)*{\mathrm{B}};(40,-12)*{\mathrm{C}};(12,4)*{\mathrm{F}};
(32,4)*{\mathrm{F}};
\ar@{-}(3,0);(18,0)
\ar@{-}(22,0);(37,0)
\ar@{-}(12,20);(28,20)
\ar@{-}(32,20);(48,20)
\ar@{-}(2,40);(18,40)
\ar@{-}(22,40);(38,40)
\ar@{-}(12,60);(27,60)
\ar@{-}(33,60);(48,60)
\ar@{-}(32,-20);(48,-20)
\ar@{-}(12,-20);(28,-20)
\ar@{-}(1,3);(9,17)
\ar@{-}(11,17);(19,3)
\ar@{-}(21,3);(29,17)
\ar@{-}(31,17);(39,3)
\ar@{-}(41,3);(49,17)
\ar@{-}(1,37);(9,23)
\ar@{-}(11,23);(19,37)
\ar@{-}(21,37);(29,23)
\ar@{-}(31,23);(39,37)
\ar@{-}(41,37);(49,23)
\ar@{-}(1,43);(9,57)
\ar@{-}(11,57);(19,43)
\ar@{-}(21,43);(29,57)
\ar@{-}(31,57);(39,43)
\ar@{-}(41,43);(49,57)
\ar@{-}(1,-3);(9,-17)
\ar@{-}(11,-17);(19,-3)
\ar@{-}(21,-3);(29,-17)
\ar@{-}(31,-17);(39,-3)
\ar@{-}(41,-3);(49,-17)
\ar@{->}(10,5);(10,-5)
\ar@{->}(30,5);(30,-5)
\endxy
$
 \end{minipage}
} 
\end{minipage}
\vspace{4ex}
\caption{Peck's Klein bottle {\em Tonnetz} and a torus variant
\label{peck2}}
\end{center}
\end{figure}

Of the four operations on triads which fix the tritone, two will exchange these two types,
just as the $P$, $L$, and $R$ operations do.  Let us call these $S$ and $F$, where $S$
moves the third note to a note in the {\em same} half of the musical circle, and $F$ 
{\em flips}
it to the appropriate note in the other half of the musical circle.  For example,
$S(\{0,n_1,N/2\}) = \{0,n_2,N/2\}$, while 
$F(\{0,n_1,N/2\}) = \{0,-n_1,N/2\} = \{0,n_2+n_3,N/2\}$.
The composite $SF = T_{N/2}$.

Examining Peck's Figure 20 (Figure~\ref{peck1}), we see that the transition from row I to row II is
done by applying $F$, since $F(\{2,4,8\}) = \{2,0,8\}$ and
$F(\{2,10,8\}) = \{2,6,8\}$. The transition from row II
to row III is by $S$:  $S(\{0,2,6\}) = \{0,4,6\}$ and 
$S(\{0,8,6\}) = \{0,10,6\}$.  To make the transition from row III back to row I we
must decide whether to apply $S$ or $F$.  If we apply $S$, we get Peck's
Klein bottle {\em Tonnetz}, as shown in Figure~\ref{peck2}.
We have labelled the 2-simplices in the top and bottom rows of this diagram 
$A$, $B$, $C$ and $D$, to
make evident the reversal which gives us a Klein bottle.

Had we chosen to join row III to row I by the operation $F$ instead, we would obtain the 
other
{\em Tonnetz} shown in Figure~\ref{peck2}.  Here, we evidently get
a torus.  Choosing either $S$ or $F$ to join each row to the next gives us eight choices
in assembling a {\em Tonnetz} from the 2-simplices of the even component of
$C(2,4,6)$. We could call them
$SSS$, $SSF$, $SFS$, $\ldots$, $FFF$, with Peck's Klein bottle {\em Tonnetz} corresponding to
$FSS$ in the ordering of simplices used in Peck's Figure 20 and our Figures~\ref{peck1} and 
\ref{peck2}.
Half of these are tori and half are Klein bottles, depending upon the number of flips $F$
used to assemble them.  If this number is even, we get a torus.  If it is odd, we get a 
Klein bottle.

In our {\em Tonnetz} $C(2,4,6)$, each row is formed into a tetrahedron, as in Figure~\ref{tetra1},
since each tritone corresponds to a single edge.  These then join one another in a circular
fashion, as shown in Figure~\ref{circleoft}.  Each of the analogs of Peck's Klein bottle {\em Tonnetz}, distinguished by the
sequences $SSS$ through $FFF$, map to $C(2,4,6)$ by collapsing the two copies of each tritone
to one.  

As an interesting side note here, the analog of the $PLR$ group will be the same for all
these {\em Tonnetze}, since their sets of 2-simplices are the same.  The differences
between them arise from the manner in which these 2-simplices are joined. The analog of the $PLR$ group 
cannot act simply transitively on the 2-simplices, since there are two components and $P$, $L$, 
and $R$ preserve the components. To get a simply transitive group action, one must use the generalized contextual group of Fiore and Satyendra \cite{Fiore2}.

\edit{Another interesting fact, discussed in the introduction and proven in Theorem~\ref{disj}, 
is that if the edge lengths share a common divisor, then the space will consist
of independent, identical components (Corollary~\ref{connected}). Peck's Fig. 20 and 21 illustrate this fact. 
In these examples,
Peck considers $C(2,4,6)$, which he shows  
consists of two copies of $C(1,2,3)$, one containing the even pitch classes (\cite[Fig. 20]{Peck})
and one containing the odd (\cite[Fig. 21]{Peck}).}

\section{The Euler Characteristic}
\label{eulchar}
\edit{The first step in our classification of spaces of triads is a
computation of the Euler characteristic. The Euler characteristic is a
geometric invariant of spaces which distinguishes the compact,
orientable surfaces without boundary. That is, two compact orientable
surfaces without boundary which have the same Euler characteristic are
topologically equivalent. For example, every torus, e.g. the
neo-Riemannian \emph{Tonnetz}, has Euler characteristic zero.
When $C(n_1,n_2,n_3)$ is non-orientable (like the M\"obius band), or has
a boundary, or is not a surface, more work will be necessary to complete
our classification of spaces of triads.}

The Euler characteristic of a finite simplicial complex is determined by
counting the $i$-simplices for each $i$.
While the number of vertices \edit{in $C(n_1,n_2,n_3)$ is simply $N=n_1
+n_2 + n_3$, computing the number of edges and faces requires more careful
counting}. We recall the definition of \edit{the} Euler characteristic
and then proceed with finding the number of edges and faces in our
simplicial complex $C(n_1,n_2,n_3)$.

\begin{definition}
The Euler characteristic, $\chi$, of a simplicial complex is:
\[
  \chi = \sum_{i \geq 0}{(-1)^in_i}
\]
where $n_i$ is the number of $i$-simplices. For a 2-dimensional complex, this reduces to
\[
  \chi = V - E + F
\]
where $V$ is the number of vertices, $E$ is the number of edges, and $F$ is the number of faces.
\end{definition}

We make the following definitions to simplify the counting of the number of edges.

\begin{definition}
Suppose $1 \leq n_1 \leq n_2 \leq n_3 < N$ are integers and $N= n_1 + n_2 +n_3$.
\begin{center}
 $X = \{\{k,k+n_1\}| \, k \in \mathbb{Z}/N\}$ \\
  $Y = \{\{k,k+n_2\}| \, k \in \mathbb{Z}/N\}$ \\
  $Z = \{\{k,k+n_3\}| \, k \in \mathbb{Z}/N\}$
  \end{center}
\end{definition}

Clearly, $X$ contains all edges of $C(n_1,n_2,n_3)$ which have length $n_1$, $Y$ contains those of length $n_2$, and \edit{$Z$ contains those of length $n_3$}. Thus, the size of the set $X$ is the number of edges of length $n_1$, and \edit{similarly} for $Y$ and $Z$. \\

\edit{We start with  some  useful lemmas.}

\begin{lemma}
 \edit{ Suppose $1 \le n_1 \le n_2 \le n_3 < N$ are integers and $N= n_1 + n_2 + n_3$. Then $n_1 < N/2$ and $n_2 < N/2$.}
\label{lemma1}
\end{lemma}

\begin{proof}
\edit{$2n_1 \leq 2n_2 \leq n_2 + n_3 < n_1 + n_2 + n_3 = N$.}
\end{proof}

\begin{lemma}
\edit{ Suppose $1 \le n_1 \le n_2 \le n_3 < N$ are integers and $N= n_1 + n_2 + n_3$. Then $n_3 = n_1 + n_2$ if and only if $n_3 = N/2$.}
\label{lemma2}
\end{lemma}

\begin{proof}
\edit{By the definition of $N$, it follows that $n_1+n_2+n_3 = N = 2n_3$ if and only if $n_1 + n_2 = n_3$.}
\end{proof}
  
\begin{lemma} Suppose $1 \leq n_1 \leq n_2 \leq n_3 < N$ are integers and $N = n_1 + n_2 + n_3$. \edit{The sets $X$ and $Y$ each have $N$ elements. The number of elements in $Z$ is 
\[
  |Z| = \left\{
      \begin{array}{ll}
         N & \quad{\textstyle if}\,\,\,n_3 \neq n_1 + n_2\\
         N/2 & \quad{\textstyle if}\,\,\,n_3 = n_1 + n_2
     \end{array}
   \right.
\]
}
\end{lemma}

\begin{proof}
\edit{Consider the list 
\[
  \{0,n_1\}, \{1,1+n_1\},\ldots,\{(N-1),(N-1)+n_1\}.
\]
Suppose $\{i,i+n_1\} = \{j,j+n_1\}$. Then $i = j$ or $i = j + n_1$. If $i = j$, then $i + n_1 = j + n_1$ and $\{i,i+n_1\}$ and $\{j,j+n_1\}$ are the same entry in the list. If, on the other hand, $i = j+n_1$, then $j = i + n_1$ and $i = i+n_1 + n_1$, so $2n_1 = 0$ in $\mathbb{Z}/N$. But this contradicts $n_1 < N/2$ in  Lemma~\ref{lemma1}. Thus $i = j$ is the only possibility, and the list has $N$ elements, so $|X|=N$. The proof of $|Y|=N$ is similar, since $n_2 <N/2$ by Lemma~\ref{lemma1}.}

\edit{The proof of $|Z| = N$ in the case $n_3 \neq n_1 + n_2$ is also similar by Lemma~\ref{lemma2}.}

\edit{To show $|Z|= N/2$ in the case $n_3 = n_1 + n_2$, note that $N=2n_3$ by Lemma~\ref{lemma2}. Then the sets $\{k,k+n_3\}$ are all different for 
$0 \leq k < n_3$, since each set $\{k,k+n_3\}$ contains only one of $0 \leq k < n_3$. Then for $n_3 \leq k < N$, we have 
\[
  \{ k,k+n_3\} = \{ n_3 +i,n_3 + i +n_3\} = \{i,i+n_3\}
\]
for some $0 \leq i < n_3$, so we are back to the first half of the list.}
\end{proof}

\begin{theorem}
The \edit{number} of edges $|\cE|$ of \edit{the} simplicial complex $C(n_1,n_2,n_3)$ is given by the following chart:
\begin{center}
\renewcommand{\arraystretch}{1.2}
\begin{tabular}{|c|c|c|c|c|}
\hline
Cases & $|X|$& $|Y|$& $|Z|$& $|X \cup Y \cup Z|$ \\
\hline
 $n_1 < n_2 < $\edit{$n_3$} and $n_3 \ne n_1 + n_2$& $N$& $N$& $N$& $3N$ \\
\hline
 $n_1 < n_2 < n_3$ and $n_3 = n_1 + n_2$& $N$& $N$& $N/2$& $5N/2$ \\
\hline
 $n_1 = n_2 < n_3$ and $n_3 \ne n_1 + n_2$& $N$& $N$& $N$& $2N$ \\
\hline
 $n_1 = n_2 < n_3$ and $n_3 = n_1 + n_2$& $N$& $N$& $N/2$& $3N/2$ \\
\hline
 $n_1 < n_2 = n_3$& $N$& $N$& $N$& $2N$ \\
\hline
 $n_1 = n_2 = n_3$& $N$& $N$& $N$& $N$ \\
\hline
\end{tabular}
\end{center}
\label{edges2}
\end{theorem}

\begin{proof}
In the first and second rows of the table, we see that the number of edges is $|X| + |Y| + |Z|$. In the third and fourth rows, the number of edges is
actually $|X| + |Z|$, since $X = Y$, and we avoid double counting. In the fifth row, we see that $Y = Z$, and hence the number of edges is only $|X| + |Y|$. In the last row, since all three sets are equal, we only count one of them to avoid triple counting.
\end{proof}

\edit{Now that we have counted the edges of $C(n_1,n_2,n_3)$, we count the faces.}
\edit{Let us call} the vertex between the edges of lengths $n_3$ and $n_1$
the {\em basepoint} of the \edit{2-simplex. 
Recall that $\cF$ is the set of all faces of $C(n_1,n_2,n_3)$.}

\begin{lemma}
\edit{Elements of
$\cF$ are of the form 
\[
\{\{k,k+n_1,k+n_1+n_2\} \,\, | \,\, k \in \mathbb{Z}/N\}
\]
or 
\[
\{\{k,k-n_1,k-n_1-n_2\} \,\, | \,\, k \in \mathbb{Z}/N \}.
\]
}
\label{defin}
\end{lemma}

\begin{proof}
\edit{By Definition~\ref{defin}, the 2-simplices are the translations and inversions of
$\{0,n_1,n_1+n_2\}$.  Let us call these {\em Type I} and {\em Type II},
respectively, with the understanding that, in exceptional circumstances, these types might overlap.
Clearly, the translations are those of the form
$\{k,k+n_1,k+n_1+n_2\}$. Just as clearly,  the inversion $I_k(i) = k-i$ 
produces those of the form $\{k,k-n_1,k-n_1-n_2\}$. }
\end{proof}

\edit{In the case of $C(3,4,5)$, the Type I simplices are the minor triads and
the Type II simplices are the major triads}. However,  in more symmetrical
situations, these types may not be distinct.

\begin{proposition}
If $n_1=n_2=n_3$, then \edit{$|\cF|$} $= N/3$.
\label{faces1}
\end{proposition}
\begin{proof}
Clearly, $N=3n_1$, and all \edit{2-simplices} have the form $\{k,k+n_1,k+2n_1\}$. This is simply the coset $k + \langle n_1 \rangle$ in $\Z/N$, so there are $N/3$ of them.
\end{proof}

\begin{proposition}
If $n_1<n_2<n_3$, then \edit{$|\cF|$} $= 2N$.
\label{faces2}
\end{proposition}

\begin{proof}
Let us start by calculating the \edit{number of} Type I \edit{2-simplices}. It is obvious
that the vertex between the edges of length $n_3$ and $n_1$ selects a
unique, well-defined value of $k$ which provides an inverse to $k
\mapsto \{k,k+n_1,k+n_1+n_2\}$, so that there are $N$ \edit{2-simplices}
of Type I. Similarly there are $N$ \edit{2-simplices} of Type II.
\edit{No 2-simplex} is of both Types simultaneously. \edit{To see this,
suppose we have a face which is of both Types I and II.
It must  contain an interval $\{a, a+n_1 \}$, so we
let $\{a,a+n_1,b\}$ be of Type I and $\{a,a+n_1,c\}$ be of Type II, 
and suppose they are equal. Since the first is Type I, $b $ must be $ a+n_1+n_2$.
This follows from the inequality
$n_1 < n_2 < n_3$,
which implies that there is a unique edge of each length
$n_1$, $n_2$ and $n_3$.
Similarly, $c $ must be $ a -n_2$ since the basepoint of this chord is
$a+n_1$, so that it can be written $\{a+n_1, (a+n_1) - n_1, c = (a+n_1) - n_1 - n_2\}$.
For these two triads
to be equal, $b=c$, that is $a-n_2 = a+n_1 + n_2$, modulo $N$. Hence, $0 = n_1 +
2n_2$, modulo $N$.   Lemma~\ref{lemma1} implies that $ 0 < n_1 + 2n_2 < 3N/2$.
The only number congruent to $0$ modulo $N$ in this interval is $N$ itself,
so we conclude that $n_1 + 2n_2 = N = n_1 + n_2 + n_3$.  But this implies
$n_2 = n_3$, contradicting our hypothesis.  Hence our faces (2-simplices) have
a well-defined Type, either I or II, when $n_1 < n_2 < n_3$.}
\end{proof}

\begin{proposition}
If $n_1=n_2<n_3$ or $n_1<n_2=n_3$, then \edit{$|\cF|$} $= N$.
\label{faces3}
\end{proposition}
\begin{proof}
\edit{Suppose $n_1=n_2<n_3$. In this case, the Type I triads can be written
$\{k,k+n_1,k+2n_1\}$. If we let $j=k+2n_1$ then the Type II triad
$\{j, j-n_1, j-2n_1\}$ is exactly the same as the Type I triad
$\{ k+2n_1, k+n_1, k\}$, so that the two Types coincide.  Seen another
way, there are two possible basepoints to a triad, since there are two edges
of length $n_1$, and the triad is of Type I with respect to one of them
and of Type II with respect to the other.  Thus, the number of faces is exactly the
number of Type I triads, and these are in one to one correspondence with
the vertices by taking the (Type I) basepoint of the triad.}

\edit{Exactly the same argument works in the case $n_1 < n_2 = n_3$.}
\end{proof}

Now that we have calculated the number of vertices, edges,
and faces for any simplicial complex $C(n_1,n_2,n_3)$, we can calculate
the Euler characteristic.

\begin{theorem}
\label{euler}
The Euler characteristic, $\chi$, of $C(n_1,n_2,n_3)$ is given by the following chart:
\begin{center}
\renewcommand{\arraystretch}{1.2}
\begin{tabular}{|c|c|c|c|c|}
\hline
Cases & $|\cV|$& $|\cE|$& $|\cF|$& $\chi$ \\
\hline
 $n_1 < n_2 < $\edit{$n_3$} and $n_3 \ne n_1 + n_2$ &  $N$ &  $3N$ &  $2N$ & $0$ \\
\hline
 $n_1 < n_2 < n_3$ and $n_3 = n_1 + n_2$ &  $N$ &  $5N/2$ & $2N$ &  $N/2$ \\
\hline
 $n_1 = n_2 < n_3$ and $n_3 \ne n_1 + n_2$ &  $N$ &  $2N$ &  $N$ &  $0$ \\
\hline
 $n_1 = n_2 < n_3$ and $n_3 = n_1 + n_2$ &  $N$ &  $3N/2$ &  $N$ &  $N/2$ \\
\hline
 $n_1 < n_2 = n_3$ &  $N$ &  $2N$ &  $N$&  $0$ \\
\hline
 $n_1 = n_2 = n_3$ &  $N$ &  $N$ &  $N/3$ &  $N/3$ \\
\hline
\end{tabular}
\end{center}
\end{theorem}

\begin{proof}
The $\cV$ column is clear.  The $\cE$ column is proved in 
Theorem~\ref{edges2}, and the $\cF$ column is proved in
Lemma~\ref{defin}
and 
Propositions~\ref{faces1}-\ref{faces3}.
\end{proof}

For further emphasis, let us check our work in Section 3. In the
\edit{chromatic} scale, we saw that the major and minor triads,
$C(3,4,5)$, form a torus. Using the chart above, we see that $C(3,4,5)$
is in the first row, so must have Euler characteristic 0. In fact, the
Euler characteristic is indeed 0, as expected \edit{from elementary
topology}.

\section{Topological Properties}
\label{topprop}
\edit{We now focus our attention on some basic topological properties of the spaces $C(n_1,n_2,n_3)$.
 Concepts like path-connectedness and homogeneity are discussed here, and these types of
 topological 
properties will prove very useful in our classification of $C(n_1,n_2,n_3)$
as a space.}
\subsection{Connectedness}
\label{connectedness}
\edit{To determine whether $C(n_1,n_2,n_3)$ is connected or not, and if not, how many components
it has, we use algebraic topology and elementary number theory.}\\

Let $H_n(X;R)$ denote the \edit{n-th} homology group of a space $X$ over
a ring $R$. Let us recall some basic facts from algebraic topology
\cite{Greenberg}. First, it is well known that $H_0(X;R) = R^d$, where
$d$  is the number of \edit{connected} components of $X$.
(\edit{Recall} that \edit{connected} components and path-components are the
same for a simplicial complex, \edit{so we may abbreviate this to {\em component}
without ambiguity.})
This allows us to compute the number
of  components of \edit{the simplicial complex
$C(n_1,n_2,n_3)$ from its 0-th homology.} In the following discussion,
consider the free abelian group \edit{$\Z[\cV]$} generated by the
vertices \edit{$\cV = \mathbb{Z}/N$ of $C(n_1,n_2,n_3)$}. This is the
group of 0-chains or 0-cycles (since every 0-chain is a cycle). Thus, we
see
\[
  H_0(X;\Z)=\frac{Z_0(X)}{B_0(X)} \cong {\displaystyle{\frac{\Z[\cV]}{\sim}}}
\]
where $\sim$ is given by \edit{all} linear combinations of
boundaries of edges. These linear combinations simply group the vertices
into groups according to the component they lie in. \edit{For example,
consider the linear combination $\partial(\{0,1\} + \{1,2\}) =
\{1\}-\{0\} + \{2\} - \{1\} = \{2\}-\{0\}$. This calculation reflects
the fact that, since there exist edges from $\{0\}$ to $\{1\}$ and
$\{1\}$ to $\{2\}$, there exists a path from $\{0\}$ to $\{2\}$, as seen
in Figure~\ref{simpcompex}.}

\edit{We now recall without proof a basic proposition 
concerning greatest common divisors
from elementary number theory. }

\begin{proposition} 
\edit{
If $m_1,\ldots,m_p$ are integers, then the subgroup of $\mathbb{Z}$ generated by 
$m_1,\ldots,m_p$ is $d\mathbb{Z}$ where $d$ is the greatest common divisor of $m_1,\ldots,m_p$.
}
\label{gcdprop}
\end{proposition}

\begin{theorem}
$H_0(C(n_1,n_2,n_3);\Z)=\Z^{\gcd(n_1,n_2,n_3)}$.
\end{theorem}
\begin{proof}
From the discussion above, we conclude that $H_0 =
\Z[\cV/{\displaystyle{\sim}}]$, where $\sim$ means `is connected by a
sequence of edges'. Now $\cV = \Z/N$ and the relation $\{k\}\sim
\{l\}$ is the equivalence relation generated by $k$ differs from
$l$ by either $n_1$, $n_2$, or $n_3$. Hence,
\[ \cV/{\displaystyle{\sim}} \cong (\Z/N)/\langle n_1,n_2,n_3 \rangle \cong \Z/{\langle n_1,n_2,n_3\rangle}
\]
since $N=n_1+n_2+n_3$ and by the third isomorphism theorem (see \cite{Weibel}).
By \edit{Proposition~\ref{gcdprop}}, we see that
\[
\Z/{\langle n_1,n_2,n_3\rangle} \cong \Z/{\gcd(n_1,n_2,n_3)}.
\]
Hence, we see that $H_0(C(n_1,n_2,n_3);\Z)=\Z^{\gcd(n_1,n_2,n_3)}$.
\end{proof}

\begin{corollary}
\label{connected}
The space $C(n_1,n_2,n_3)$ is connected if and only if $\gcd(n_1,n_2,n_3) = 1$.
\end{corollary}
\begin{proof}
This follows immediately from the above theorem, noting that $H_0$ detects the number of connected components of a \edit{simplicial complex}.
\end{proof}

\edit{As an example, let us apply the above corollary to $C(3,4,5)$. Since $\gcd(3,4,5) = 1$, 
 Corollary~\ref{connected} implies that the space $C(3,4,5)$ is connected. This is just confirmation
 of what we already knew: the torus is connected.}

\begin{theorem}
\label{disj}
The simplicial complex $C(dn_1,dn_2,dn_3)$ is the disjoint union of $C(n_1,n_2,n_3)$. Specifically, 
\[
C(dn_1,dn_2,dn_3) = \coprod_{d}C(n_1,n_2,n_3).
\]
\end{theorem}
\begin{proof}
A 2-simplex with edges of \edit{lengths} $dn_1$, $dn_2$, and $dn_3$, must
connect vertices which differ by multiples of $d$. Hence, the vertices
in any simplex must lie in a \edit{coset of $\langle d \rangle $ in $\Z/N$}. Thus, the
set of all simplices can be decomposed into a disjoint union over
\edit{these cosets}.
\end{proof}

Corollary~\ref{connected} shows that if the $\gcd(n_1,n_2,n_3) = 1$,
then $C(n_1,n_2,n_3)$ cannot be decomposed further. Hence, if
$\gcd(n_1,n_2,n_3)=1$, Theorem~\ref{disj} gives the decomposition of
$C(dn_1,dn_2,dn_3)$ into its connected components.  

Since the analogues of the neo-Riemannian $P$, $L$, and $R$ operations flip triads across edges, they cannot take a triad in one component to a triad in a different
component. Thus, when $C(n_1,n_2,n_3)$ is disconnected, the group generated by the $P$, $L$, and $R$ analogues cannot act simply transitively on the triads in $C(n_1,n_2,n_3)$, and this group is a proper subgroup of Fiore--Satyendra's generalized contextual group associated to $\{0,n_1,n_1+n_2\}$ in \cite{Fiore2}.  For example, in $C(2, 4, 6)$, no combination of the $P$, $L$, and $R$ analogues takes $\{0, 2, 6\}$ to $\{1, 3, 7\}$, though the contextual operation $Q_1$ certainly does.

\subsection{Homogeneity}

\edit{In general, homogeneity means that every point is the same as every other
in some appropriate sense. We say $C(n_1,n_2,n_3)$ is {\em homogeneous for $n$-simplices}
if for any two $n$-simplices $\sigma$ and $\sigma'$ there is an automorphism
of the simplicial complex $C(n_1,n_2,n_3)$ which maps $\sigma$ to $\sigma'$.
 For our simplicial complexes, we can ask for such
homogeneity among vertices, among edges, or among 2-simplices.  Since
$T_{j-i}(i) = j$, we clearly have homogeneity among vertices. This means the
analysis of any one vertex suffices to describe them all.}

\edit{We also have homogeneity among 2-simplices using the full $T/I$ group.
For any two 2-simplices of the same Type, there is a translation taking one to the other.
Inversion, on the other hand, converts a simplex from one Type to the other.
Hence, given any two 2-simplices $A$ and $B$, there is a member of the full $T/I$ group
which sends $A$ to $B$.  Therefore, our analysis of any one 2-simplex suffices to
describe all 2-simplices.}

\edit{However, we do not have homogeneity among the 1-simplices.
For example, cylinders and M\"obius
bands (e.g., Mazzola's Harmonic strip)
have boundary edges and interior edges, which cannot be sent to one another by
an invertible map of simplicial complexes, since such a map will preserve the number
of 2-simplices an edge is contained in.  Similarly, in the circle of tetrahedra boundaries,
there are edges (corresponding to tritones)
which lie in four 2-simplices and edges which lie in only two 2-simplices,
so there is no invertible simplicial map which can send the first sort to the second,
or vice versa.}

\section{Classification of $C(n_1,n_2,n_3)$ as a space.}
\label{classpa}

We are now ready to classify all of the possible
spaces of triads. We begin by considering the 
number of simplices which contain a given dege, since this allows us to
distinguish between surfaces with boundary, surfaces without boundary,
and non-surfaces by Theorem~\ref{edges}. Lemma~\ref{halftotlemma1}
shows that an edge can be contained in only one or two 2-simplices
unless $n_3 = N/2$, since Lemma~\ref{lemma1} eliminates the possibility 
that $n_1$ or $n_2$ can be $N/2$. When $n_3 \neq N/2$, we
obtain a surface with or without boundary, in Theorems~\ref{tori}
through \ref{nnnplusk}. Finally, Theorem~\ref{circleoftetra} handles the case when
$n_3 = N/2$, in which $C(n_1,n_2,n_3)$ is not a surface. Note that
Lemma~\ref{lemma2} says that $n_3 = N/2$ is equivalent to $n_3 = n_1 + n_2$.

\subsection{When is $C(n_1,n_2,n_3)$ a surface without boundary?}

\label{surfwobdy}
\edit{The first step in our classification of $C(n_1,n_2,n_3)$ as a space begins with surfaces
without boundary, or just surfaces. We define this notion, then investigate the pertinent cases.}
We prove if $C(n_1,n_2,n_3)$ is a surface without boundary, then it is a disjoint 
union of tori or a disjoint union of spheres.

\begin{definition}
A \emph{surface without boundary} (or just \emph{surface}) is a \edit{second countable} Hausdorff, topological space in which every point has an open neighborhood homeomorphic to some open subset of $\R^2$.
\end{definition}

\edit{
These include the sphere, which is topologically equivalent to the boundary of a 
tetrahedron, and the torus.  The requirement that a surface be second countable and Hausdorff
serves to eliminate pathological examples.
}

\begin{definition}
A \emph{surface with boundary} is a \edit{second countable} Hausdorff, topological
space whose points have neighborhoods which are homeomorphic to open
subsets of $\R^2_+=\{(x,y)\in\R^2:y\geq0\}$.  
\end{definition} 

Points
which have an open neighborhood homeomorphic to an open subset of $\R^2$
are called {\em interior points}. The other points are called {\em boundary points}.
\edit{
Examples include the 2-simplex and rectangle, whose boundary points are their
perimeters, the cylinder, whose boundary consists of two circles, and
the M\"obius band, whose boundary is a single circle.
}

There are two key properties required for a 2-dimensional simplicial
complex to be a surface: (i) \edit{each edge must lie in exactly two
2-simplices}, and (ii) at each vertex, the \edit{2-simplices} containing
the vertex can be cyclically ordered (sharing an edge with each
neighbor). To be a surface with boundary: (i) \edit{each} edge must be
contained in one or two \edit{2-simplices}, and (ii) at each vertex, the
\edit{2-simplices} containing the vertex can be placed in
\edit{
either a cyclic order (in which case the vertex is an interior point)
or a linear order (in which case the vertex is a boundary point)}. 

\edit{We begin by analyzing the first property of a surface noted above.} We \edit{prove} the following lemma in order to make the identification of edges simpler. Let $\{i,j,k\}=\{1,2,3\}$, so that $\{n_i,n_j,n_k\}=\{n_1,n_2,n_3\}$. \edit{Recall that we assume $n_1 +n _2 + n_3 = N$ and $1 \leq n_1 \leq n_2 \leq n_3 < N$.}

\begin{lemma} 
\label{halftotlemma1}
If $n_i \neq N/2$, then each edge of length $n_i$ \edit{in $C(n_1,n_2,n_3)$ is contained in exactly one or two 2-simplices}.
\end{lemma}

\begin{proof}
Since $n_i \neq N/2$, an edge of length $n_i$ can be written uniquely as $\{a,a+n_i\}$. \edit{The
2-simplices} containing this edge must have vertices $\{a,a+n_i,b\}$. We see that either $b= a+n_i
+ n_j$ or $b = a + n_i + n_k$. Hence, there exist \edit{two 2-simplices containing
this edge} if $n_j \neq n_k$ and \edit{one such 2-simplex} if $n_j = n_k$.
\end{proof}


\begin{lemma}
\label{halftotlemma2}
If $n_i = N/2$, then \edit{each} edge of length $n_i$ \edit{in $C(n_1,n_2,n_3)$ is contained in exactly two or four 2-simplices}.
\end{lemma}

\begin{proof}
We see that $i = 3$ \edit{(by Lemma~\ref{lemma1})} and $\{a,a+n_3\}$ no longer determines which vertex is $a$. The edge divides the
musical circle in half and the third vertex of the \edit{2-simplex} containing it can be in either half. \edit{If} $n_1 \neq n_2$,
there are two choices on either side of the circle, yielding a total of four \edit{2-simplices} which contain the edge. On the other
hand, if $n_1 = n_2$, then there is a unique choice on each side, giving only two such \edit{2-simplices}.
\end{proof}

The above calculations allow us to determine when $C(n_1,n_2,n_3)$ is a surface with or without boundary. If an edge is contained in more than two \edit{2-simplices}, then the space \edit{will} have singularities, and is not a surface. If an edge \edit{is} only contained in one \edit{2-simplex}, then the space \edit{will} have a boundary. 

\begin{theorem}
\label{edges}
The number of \edit{2-simplices in $C(n_1,n_2,n_3)$} containing an edge of length $n_i$ is determined by the following chart:
\begin{center}
  \begin{tabular}{|c|c|c|}
    \hline
    & $n_i \neq N/2$ & $n_i = N/2$ \\
    \hline
    $n_j = n_k$ & 1 & 2 \\
    \hline
    $n_j \neq n_k$ & 2 & 4 \\
    \hline
  \end{tabular}
\end{center}
\end{theorem}

\begin{proof}
\edit{This follows} directly from \edit{Lemmas \ref{halftotlemma1} and \ref{halftotlemma2}} and their proofs.
\end{proof}

\edit{The other key property of a surface without boundary is the ordering of the 2-simplices
in a cyclic fashion around each vertex.}
In order to investigate \edit{this} criterion, we \edit{list} all the possible \edit{2-simplices in
$C(n_1,n_2,n_3)$} which \edit{could}
contain the vertex $\{a\}$. 

\begin{enumerate}
\item $\{a,a+n_1,a+n_1+n_2\}$
\item $\{a, a+n_1,a+n_1+n_3\}$
\item $\{a,a+n_2,a+n_1+n_2\}$
\item $\{a,a+n_2,a+n_2+n_3\}$
\item $\{a,a+n_3, a+n_1+n_3\}$
\item $\{a, a+n_3,a+n_2+n_3\}$
\end{enumerate}

When we have \edit{additional} relations among the numbers $n_1,n_2$,
and $n_3$, we expect \edit{some elements of} this list to be redundant.
Observe that we can arrange the \edit{2-simplices} in this list
\edit{cyclically} as shown in Figure~\ref{hexagon}.
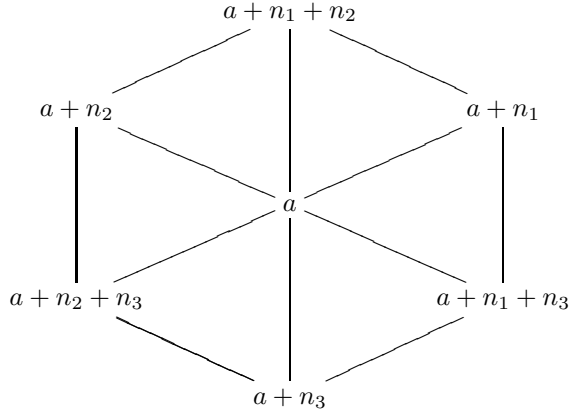
\begin{figure}
\[
\xymatrix{
& a+n_1+n_2 
\ar@{-}[dl]
\ar@{-}[dd]
\ar@{-}[dr]
& 
\\
a+n_2
\ar@{-}[dr]
&&
a+n_1
\ar@{-}[dl]
\\
&
a
&
\\
a+n_2+n_3
\ar@{-}[uu]
\ar@{-}[ur]
\ar@{-}[dr]
&&
a+n_1+n_3
\ar@{-}[uu]
\ar@{-}[lu]
\\
&
a+n_3
\ar@{-}[lu]
\ar@{-}[ru]
\ar@{-}[uu]
&
\\
}
\]
\caption{The \edit{cyclic} placement of 2-simplices around the vertex $a$.}
\label{hexagon}
\end{figure}
In this picture, we see each of the six \edit{2-simplices} in the list
above placed around the vertex $\{a\}$. Hence, if all six
\edit{2-simplices} are distinct, the \edit{cyclic} ordering of
\edit{2-simplices} containing each vertex can be accomplished. 
(Figures~\ref{tetra1} and ~\ref{band}
show what can happen when they are not all distinct.)

\edit{We will now show that all the spaces $C(n_1,n_2,n_3)$ which
are surfaces without boundary are orientable.
(The cyclical
ordering does not guarantee this.)
The Euler characteristic (which we have already computed) then
completely determines the topological type of the surface.
To do this we will assign consistent orientations to all 2-simplices in
$C(n_1,n_2,n_3)$.  We begin by noting a  relation between adjacent 2-simplices.
}

\begin{lemma}
\label{Type}
If $n_1<n_2<n_3 \neq N/2$, then \edit{2-simplices} which share an edge must be of opposite Type.
\end{lemma}

\begin{proof}
A \edit{2-simplex} is considered to be \edit{Type I} if the edge \edit{lengths have cyclic order
$n_1,n_2,n_3$ when read in clockwise order}. \edit{A 2-simplex is considered to be Type II if the
edge lengths have cyclic order $n_1,n_3,n_2$ when read in clockwise order}. Since $n_3 \neq N/2$,
each edge of length $n_i$ can be written uniquely as $\{a,a+n_i\}$. \edit{By
Lemma~\ref{halftotlemma1}}, there are \edit{two 2-simplices} sharing this edge,
$\{a,a+n_i,a+n_i+n_j\}$ and
$\{a,a+n_i,a+n_i+n_k\}$. However, this first \edit{2-simplex} has edge lengths $(n_i,n_j,n_k)$ and the second has $(n_i,n_k,n_j)$. Hence, whenever we have a \edit{2-simplex} of either Type, any \edit{2-simplex} sharing an edge with it will be of \edit{the} opposite \edit{Type}.
\end{proof}

\edit{This generalizes the duality of major and minor triads mentioned
earlier in Lemma~\ref{defin}. On the other hand, if any two edge lengths are equal, then the two
types coincide, as shown in the proofs of Propositions~\ref{faces1} and~\ref{faces3}.}

\begin{figure}
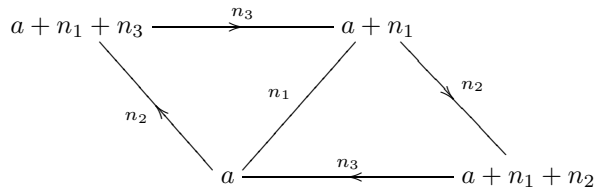

\[
\xy (0,20)*{a+n_1+n_3};(40,20)*{a+n_1};(20,0)*{a};(60,0)*{a+n_1+n_2};
(32,3)*{};

\ar^>{n_3} (10,20);(22,20) 
\ar@{-} (22,20);(34,20) 
\ar@{-<}_>{n_2} (3,18);(10.5,9.5)
\ar@{-} (10.5,9.5);(18,1)
\ar@{-}^>{n_1} (22,1);(29.5,9.5)
\ar@{-} (29.5,9.5);(37,18)
\ar@{-<}^>{n_3} (22,0);(36,0)
\ar@{-} (36,0);(50,0)
\ar@{->}^>{n_2} (43,18);(50,10.5)
\ar@{-} (50,10.5);(57,3)
\endxy
\]

\caption{The orientation of our \edit{2-simplices}. Note the cancellation along the common edge.}
\label{orient}
\end{figure}


\begin{theorem}
\label{tori}
If $n_1<n_2<n_3 \neq N/2$, then $C(n_1,n_2,n_3)$ is the disjoint union of tori.
\end{theorem}

\begin{proof}
Theorem~\ref{edges} implies that our space satisfies the condition
that \edit{each} edge \edit{is} contained in exactly \edit{two 2-simplices}. 
The 
\edit{cyclic} ordering of the \edit{2-simplices} containing a vertex $\{a\}$
is evident in Figure~\ref{hexagon}, once we observe that the condition
$n_1 < n_2 < n_3$ implies that all six of the 2-simplices containing vertex $a$
are distinct.  Thus, $C(n_1,n_2,n_3)$ is a surface.

We now orient 
\edit{all the 2-simplices} in a consistent fashion, \edit{by taking
the orientation of the edges to be by length, from least to greatest.
}
This is a consistent orientation by
Lemma~\ref{Type}, as one can easily see in
Figure~\ref{orient}.
\edit{
The figure shows only the comparison between simplices which share an
edge of length $n_1$,  but the other two cases are similar.
Thus, $C(n_1,n_2,n_3)$ is an orientable surface.
}

Now recall that the Euler characteristic of $C(n_1,n_2,n_3)$ is
0 by Theorem~\ref{euler}, \edit{and, if $d = \gcd(n_1,n_2,n_3)$
then by Theorem~\ref{disj}, this Euler characteristic is $d$ times
the Euler characteristic of each of the components,
so that each component of $C(n_1,n_2,n_3)$ must
also have Euler characteristic 0}.
\edit{By the classification
of compact, connected, orientable surfaces (described for example
in~\cite{Pressley})}, we see that any compact, connected, orientable surface with
Euler characteristic 0 is a torus. Hence, $C(n_1,n_2,n_3)$ is
the disjoint union of tori.
\end{proof}

When the \edit{edge lengths of the 2-simplices are not all distinct}, we can still have a surface.

\begin{theorem}
\label{2-sphere}
If $n_1=n_2$ and $n_3=N/2$, then \edit{$n_3 = 2n_1$ and} $C(n_1,n_1,2n_1)$ is
the disjoint union of $n_1$ spheres.
\end{theorem}

\begin{proof}
The fact that $n_3 = 2n_1$ follows directly from Lemma~\ref{lemma2}.
It follows from Theorem~\ref{disj} that we only need to compute $C(1,1,2)$. In $C(1,1,2)$, we have 4 faces: $\{0,1,2\}$, $\{0,1,3\}$, $\{0,2,3\}$, and $\{1,2,3\}$. Gluing these faces along common edges, we see the space $C(1,1,2)$ in
Figure~\ref{tetra1}.

\edit{These four 2-simplices are the boundary of the
tetrahedron with vertices $\{0,1,2,3\}$}, and the boundary of a tetrahedron
is topologically equivalent to a (2-dimensional) sphere.
We note the nice ordering around all the vertices, just as described above.
\end{proof}

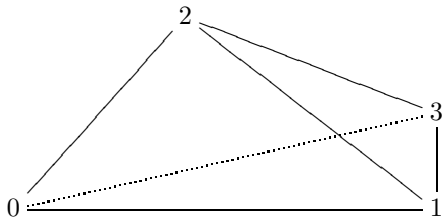
\begin{figure}
\[
\xymatrix{
&&
2
\ar@{-}[ddll]
\ar@{-}[drrr]
\ar@{-}[ddrrr]
&& \\
&&&&&
3
\ar@{-}[d] 
 \\
0
\ar@{-}[rrrrr]
\ar@{.}[rrrrru]
&&&&&
1 \\
}
\]
\caption{A picture of $C(1,1,2)$, a tetrahedron boundary.}
\label{tetra1}
\end{figure}

\begin{theorem}
\edit{If $C(n_1,n_2,n_3)$ is a surface without boundary, then it is either a disjoint union of tori
or a disjoint union of spheres}.
\label{surfall}
\end{theorem}
\begin{proof}
\edit{Suppose $C(n_1,n_2,n_3)$ is a surface without boundary. Since $C(n_1,n_2,n_3)$ is a surface without boundary, each edge is contained in exactly two 2-simplices. Hence, by Theorem~\ref{edges}, each edge length $n_i$ satisfies exactly one of the following two Conditions.}
\begin{enumerate}
  \item \edit{$n_i = N/2$ and $n_j = n_k$}
  \item \edit{$n_i \neq N/2$ and $n_j \neq n_k$}
\end{enumerate}
\edit{By Lemma~\ref{lemma1}, $n_1 < N/2$ and $n_2 < N/2$, so Condition (2) 
applies to $n_1$ and $n_2$, and we have $n_1 \neq n_3$ and $n_2 \neq n_3$. This, 
combined with our standing assumption $1 \leq n_1 \leq n_2 \leq n_3 < N$, implies 
that $1 \leq n_1 \leq n_2 < n_3 < N$.}

\edit{The integer $n_3$, on the other hand, is either equal to $N/2$ 
or not equal to $N/2$. If $n_3 = N/2$, then we are in Condition (1), 
and $n_1 = n_2$, so Theorem~\ref{2-sphere} applies and $C(n_1,n_2,n_3)$ is 
a disjoint union of spheres. If $n_3 \neq N/2$, then we are in 
Condition (2), and $n_1 \neq n_2$, so we have $1 \leq n_1 < n_2 < n_3 
\neq N/2$, precisely the situation of Theorem~\ref{tori}, and 
$C(n_1,n_2,n_3)$ is a disjoint union of tori.}
\end{proof}

\subsection{Surfaces with Boundary}
\label{surfwbdy}
The next simplest kind of space is a surface with boundary. In order for
$C(n_1,n_2,n_3)$ to have a boundary, \edit{there must be an edge contained in
exactly one 2-simplex}. 
\edit{In this section,
we classify all such spaces whose connected components are either
a 2-simplex, cylinder, or M\"obius band. As
before, we consider a general vertex $a$ and compute the 2-simplices
which contain it. We attach the 2-simplices along equal edges and
determine the resulting spaces.}

\begin{theorem}
\label{3cases}
There are only 3 cases in which $C(n_1,n_2,n_3)$ is a surface with boundary:
\begin{enumerate}
\item $C(n,n,n)$.
\item $C(n,n,n+k)$, with $n+k \neq N/2$.
\item $C(n,n+k,n+k)$.
\end{enumerate}

\end{theorem}

\begin{proof}
\edit{
Suppose $C(n_1,n_2,n_3)$ is a surface without boundary.
We must have at least one edge which is contained in only one 2-simplex,
and we must not have any edges contained in more than two 2-simplices.
Theorem~\ref{edges} implies that we must have $n_i \neq N/2$ and
$n_j = n_k$ for some arrangement of the indices $1$, $2$, and $3$ as $i$, $j$ and $k$.
This implies that either $n_1 = n_3$, $n_1 = n_2$ or $n_2 = n_3$.
If $n_1 = n_3$ we are in case (1) since $n_1 \leq n_2 \leq n_3$ forces
all three to be equal.  Since $N=3n$, the condition that $n_2 \neq N/2$
is then automatic.  If $n_1 = n_2 \neq n_3$ then we are in case (2),
with the condition $n_3 \neq N/2$.  (We will find it convenient to write
$n_1 = n_2 = n$ and $n_3 = n+k$ with $k > 0$ in this case.)
If $n_1 \neq n_2 = n_3$, then we have case (3), with the condition $n_1 \neq N/2$
following from Lemma~\ref{lemma1}.  Again, we find it convenient
to write $n_1=n$ and $n_2 = n_3 = n+k$.
}

So far, we have proved that if $C(n_1,n_2,n_3)$ is a surface with boundary, then it must be of the form (1), (2), or (3). In Theorems~\ref{triangle} and \ref{nnnplusk}, we prove that each of these is indeed a surface with boundary.
\end{proof}

\edit{In all three of these cases, it easily follows from Lemma~\ref{defin}
that there is no distinction between Type I and Type II 2-simplices.}

\begin{theorem}
\label{triangle}
The space $C(n,n,n)$ is the disjoint union of n copies of a \edit{2-simplex}.
\end{theorem}
\begin{proof}
By Theorem~\ref{disj}, 
\edit{
$C(n,n,n)$ is simply the disjoint union of $n$ copies of $C(1,1,1)$.
}
But $C(1,1,1)$ is \edit{the 2-simplex with vertices $0,1,2 \in \mathbb{Z}/3$.} 
\end{proof}

\edit{In the remaining two cases from Theorem~\ref{3cases},  the 
parity of $N$ completely determines the space.} We may assume that 
the $\gcd(n,k)=1$, otherwise by Theorem~\ref{disj}, we'd be looking
at a disjoint union of component spaces.

\begin{theorem}
\label{nnnplusk}
\edit{Assume that the $\gcd(n,k)=1$ and $n+k \neq N/2$. The spaces $C(n,n,n+k)$ and $C(n,n+k,n+k)$ are cylinders if
$N$ is even and M\"obius bands if $N$ is odd.}
\end{theorem}

Before we begin the proof, 
\edit{
we use the Snake Lemma from homological algebra to prove a
number theoretic fact which we need.
}

\begin{lemma}[(Snake Lemma)]
\label{snake}
Let
\[
\xymatrix{
0
\ar[r]
&
A
\ar@{->}[r]^f
\ar@{->}[d]^a
&
B
\ar@{->}[r]^g
\ar@{->}[d]^b
&
C
\ar@{->}[r]
\ar@{->}[d]^c
&
0
\\
0
\ar@{->}[r]
&
D
\ar@{->}[r]^{f'}
&
E
\ar@{->}[r]^{g'}
&
F 
\ar[r]
&
0
\\
}
\]
be a commutative diagram of modules with exact rows.
Then there exists an exact sequence relating the kernels and cokernels of $a$, $b$, and $c$:
\begin{equation}
\label{snakeseq}
0 \lra \Ker a \lra \Ker b \lra \Ker c \stackrel{\delta}{\lra} \Cok a \lra \Cok b \lra \Cok c \lra 0.
\end{equation}
\end{lemma}

\begin{proof}
See \cite{Weibel}.
\end{proof}

\begin{figure}
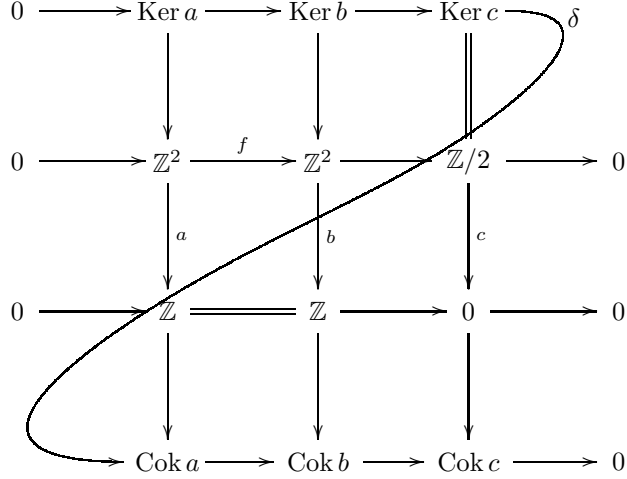

\[
\xy (40,0)*{\Cok b};(60,0)*{\Cok c};(0,60)*{0};
(0,20)*{0};(20,20)*{\Z};(40,20)*{\Z};(60,20)*{0};(80,20)*{0};(74,59)*{\delta};
(0,40)*{0};(20,40)*{\Z^2};(40,40)*{\Z^2};(60,40)*{\Z/2};(80,40)*{0};
(20,60)*{\Ker a};(40,60)*{\Ker b};
(80,0)*{0};
(60,60)*{\Ker c};(20,0)*{\Cok a}; (65,60)*{};(14,0)*{};
**\crv{(-10,0)&(20,25)&(60,40)&(80,60)}
?<*\dir{<} 
\ar@{->} (20,57);(20,43)
\ar@{->} (20,37);(20,23)^{a}
\ar@{->} (20,17);(20,3)

\ar@{->} (40,57);(40,43)
\ar@{->} (40,37);(40,23)^{b}
\ar@{->} (40,17);(40,3)

\ar@{=} (60,57);(60,43)
\ar@{->} (60,37);(60,23)^{c}
\ar@{->} (60,17);(60,3)

\ar@{->} (3,60);(15,60)
\ar@{->} (25,60);(35,60)
\ar@{->} (45,60);(55,60)

\ar@{->} (3,40);(17,40)
\ar@{->} (23,40);(37,40)^{f}
\ar@{->} (43,40);(55,40)
\ar@{->} (65,40);(77,40)

\ar@{->} (3,20);(17,20)
\ar@{=} (23,20);(37,20)
\ar@{->} (43,20);(57,20)
\ar@{->} (63,20);(77,20)

\ar@{->} (25,0);(34,0)
\ar@{->} (46,0);(54,0)
\ar@{->} (66,0);(77,0)
\endxy
\]
\caption{The Snake Lemma used in Lemma~\ref{gcd}}
\label{snakediagram}
\end{figure}

The situation we need, shown in Figure~\ref{snakediagram}, has
\[
f = \begin{bmatrix}3 & 1 \\ 1 & 1 \end{bmatrix}, 
\qquad
b = \begin{bmatrix} n & k \end{bmatrix}, 
\qquad
{\mbox{ and so }}
\quad
a = \begin{bmatrix} 3n+k & n+k \end{bmatrix}.
\]

Therefore, $\Cok{a} = \Z/{\gcd(3n+k,n+k)}$ and
$\Cok{b} = \Z/{\gcd(n,k)}$, \edit{by Proposition~\ref{gcdprop}}.
The $\Z/2$ in Figure~\ref{snakediagram} occurs because the \edit{order of the cokernel of a homomorphism between
free abelian groups is the determinant of the matrix representing the homomorphism 
(when that determinant is nonzero).  Here, the determinant of $f$
is ${\mathrm{det}}(f) = 3\cdot 1 - 1\cdot 1 = 2$.
}

\edit{The Snake Lemma exact sequence (\ref{snakeseq}), is then  the following long exact sequence.}
\[
0 \lra \Z \lra \Z \lra \Z/2 \stackrel{\delta}\lra \Z/{\gcd(3n+k,n+k)} \lra \Z/{\gcd(n,k)}
\lra 0
\]

\begin{lemma} 
\label{gcd}
Let $n_0 = n/\gcd(n,k)$ and $k_0 = k/\gcd(n,k)$.  Then
\[
\gcd(3n+k,n+k) = 
\left\{
\begin{array}{lcl}
2 \gcd(n,k) & & {\mbox{if }}n_0 {\mbox{ and }} k_0 {\mbox{ both odd}}\\
\gcd(n,k) &  &{\mbox{otherwise }}
\end{array}\right.
\]
\end{lemma}

\begin{proof}
The image of $\delta$ either has order $1$ or order $2$,
and correspondingly, $\gcd(3n+k,n+k)$ is either $1$  or $2$
times $\gcd(n,k)$.
By exactness of the sequence, the ratio is $2$
if and only if the map $\Z \lra \Z$ of kernels is onto.  Since the
kernel of $b$ is generated by ${\begin{bmatrix} k_0 & -n_0
\end{bmatrix}}^{T}$, this is equivalent to asking whether or not we can solve
\[
\begin{bmatrix} 3 & 1 \\ 1 & 1 \end{bmatrix}
\begin{bmatrix} x \\ y \end{bmatrix}
=
\begin{bmatrix} k_0 \\ -n_0 \end{bmatrix}
\]
for \edit{integers} $x$ and $y$. Row reducing \edit{the} matrix \edit{for $f$} yields the following augmented matrix:
\[
\left[\begin{array}{cr|l} 1 & 1\,\, & \,\,-n_0 \\ 0 & -2\,\, & \,\,k_0+3n_0 \end{array}\right]\mbox{.}
\]
This is solvable in integers, and hence the ratio is 2, if and only if
$k_0+3n_0$ is even. Therefore, the $\gcd(3n+k,n+k) = 2 \gcd(n,k)$ if and
only if $3n_0+k_0$ is even, and $\gcd(3n+k,n+k)= \gcd(n,k)$ if and only
if $3n_0+k_0$ is odd.
\end{proof}

\edit{
We are now prepared to finish the proof of Theorem~\ref{nnnplusk}, completing
the proof of Theorem~\ref{3cases}.
}

\begin{proof}
\edit{
By Theorem~\ref{edges} and the proof of Theorem~\ref{3cases}, each edge is contained in either one or two 2-simplices,
so that the condition on edges for a surface with boundary is satisfied.
Next, we will show that the 2-simplices containing each vertex have a linear
order, verifying the other condition for a surface with boundary, and
showing that all the vertices are on the boundary.
In the process we will  observe that the 2-simplices are arranged as in
Mazzola's harmonic strip, and hence form either a M\"obius band or cylinder.
}

\edit{
We consider  $C(n,n,n+k)$ first.
Eliminating redundancies from the list of 2-simplices following the proof of
Theorem~\ref{edges}, we find that
the following 2-simplices contain the vertex $a$.
}
\begin{enumerate}
\item $\{a,a+n,a+2n\}$
\item $\{a,a+n,a+2n+k\}$
\item $\{a,a+n+k,a+2n+k\}$
\end{enumerate}
\edit{(Recall that  $n_1=n_2=n$ and $n_3=n+k$ here.)}

We can arrange these around the vertex $a$ in a consistent, \edit{linear} fashion,
\edit{as shown in  Figure~\ref{band}.} So we now know that $C(n,n,n+k)$ is a surface
without boundary.
We also note the side lengths, as this will be useful in our analysis. 
\begin{figure}
\[
\xymatrix{
\cdots
&
a+2n
\ar@{-}[rr]^{n+k}
\ar@{-}[ddr]^{n}
&&
a
\ar@{-}[rr]^{n+k}
\ar@{-}[ddl]^{n}
\ar@{-}[ddr]^{n}
&&
a+n+k
&
\cdots
\\
&&&&&&
\\
&
\cdots
&
a+n
\ar@{-}[rr]^{n+k}
&&
a +2n + k
\ar@{-}[uur]^n
&
\cdots &
\\
}
\]
\caption{\edit{2-simplices} placed around the vertex $a$ in $C(n,n,n+k)$.}
\label{band}
\end{figure}
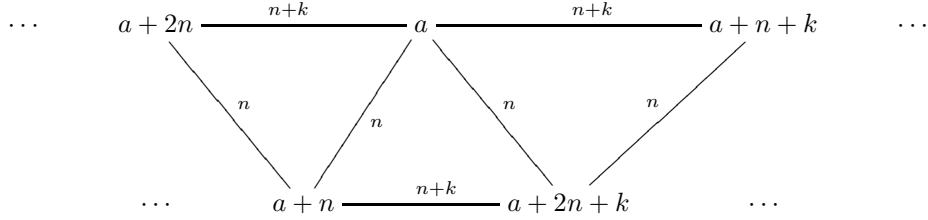

Note that the boundary edges are those of length $n+k$,
since each edge of length $n_3 = n+k \neq N/2$ 
lies in only one 2-simplex, as
shown in Theorem~\ref{edges}.

The space $C(n,n,n+k)$ can be formed by starting with a \edit{2-simplex}
$\{a,a+n,a+2n+k\}$ and moving right in Figure~\ref{band} until this
\edit{2-simplex} reappears. (Recall that $\gcd(n,n+k)=1$ implies that
$C(n,n,n+k)$ is connected by Corollary~\ref{connected}, so we know this
will occur.)

The only question remaining is
whether the edges are joined preserving or reversing orientation,
when we return to the starting 2-simplex.
This is the same as asking whether the boundary has one or two components,
since we already know that $C(n,n,n+k)$ is a surface without boundary by~\cite{Pressley}.

Two vertices lie in the same component of the boundary if they can be connected by repeatedly
adding $n+k$ to one to get the other.  
Suppose $N$ is even.
Since $\gcd(N,n+k) = \gcd(3n+k,n+k)=2$, 
there are two cosets for the subgroup generated by $n+k$ in $\Z/N$,
and hence the boundary has two components.  This implies that the
space is a cylinder.  By Lemma~\ref{gcd}, this occurs when $n$ and $k$ are both
odd, so that $N = 3n+k$ is even.

Suppose $N$ is odd.  Since $n$ and $k$ cannot both be even, by
our assumption that $\gcd(n_1,n_2,n_3) = \gcd(n,k) = 1$, the only possibility
is that exactly one of $n$ and $k$ is odd, since
$N$ is odd. In this case the subgroup of $\Z/N$ generated by $n+k$ is
the whole group, so the boundary has only one component, and the space is
a M\"obius band.

For $C(n,n+k,n+k)$, \edit{we again consider the list of 2-simplices
which contain the vertex $a$, eliminating the redundant
cases as before.}

\begin{enumerate}
\item $\{a,a+n,a+2n+k\}$
\item $\{a,a+n+k,a+2n+2k\}$
\item $\{a,a+n+k,a+2n+k\}$
\end{enumerate}
Once again, we arrange these around the vertex $a$, noting that each
edge of length $n$ \edit{is contained in exactly one 2-simplex} by Theorem~\ref{edges}, as
shown in Figure~\ref{band2}.

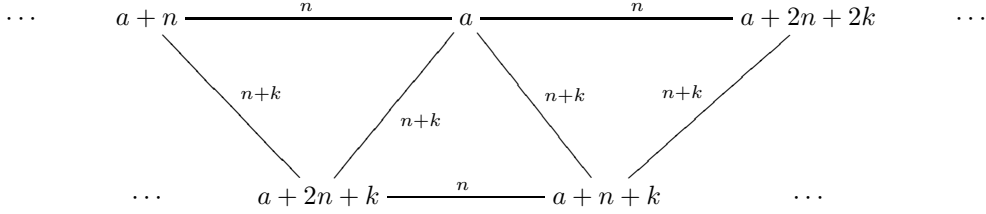
\begin{figure}
\[
\xymatrix{
\cdots
&
a+n
\ar@{-}[rr]^{n}
\ar@{-}[ddr]^{n+k}
&&
a
\ar@{-}[ddl]^{n+k}
\ar@{-}[drd]^{n+k}
\ar@{-}[rr]^{n}
&&
a+2n+2k
&
\cdots
\\
&&&&&&
\\
&
\cdots
&
a+2n+k
\ar@{-}[rr]^{n}
&&
a+n+k
\ar@{-}[ruu]^{n+k}
&
\cdots
&
\\
}
\]
\caption{\edit{2-simplices} placed around the vertex $a$ in $C(n,n+k,n+k)$.}
\label{band2}
\end{figure}

Two vertices lie in the same component of the boundary if they can be
connected by repeatedly adding $n$ to one to get the other.
Hence, we now consider $\gcd(N,n) = \gcd(3n+2k,n)$
to see how many components the boundary has.

Using \edit{the fact that $ \gcd(3b+a,b) = \gcd(a,b) $}, 
we see that $\gcd(3n+2k,n)=\gcd(2k,n)$. 
\edit{
Since $\gcd(n_1,n_2,n_3) = \gcd(n,k) = 1$,
}
\[
\gcd(N,n) =
\gcd(3n+2k,n) = 
\gcd(2k,n) = 
\left\{
\begin{array}{ll}
1 & \quad  {\mbox{if $n$ is odd}}\\
2 & \quad  {\mbox{if $n$ is even}.}
\end{array}\right.
\]
\edit{Since $N$ and $n$ have the same parity, we have one boundary component,
and hence a M\"obius band, when $N$ is odd, and two boundary components, and hence a 
cylinder, when 
$N$ is even.}
\end{proof}

\subsection{Not a surface}
\label{nonsurf}
Using the conditions on the $n_i$ of Theorem~\ref{edges}, 
Theorems~\ref{surfall} and \ref{3cases} imply
we have now classified all of the cases when $C(n_1,n_2,n_3)$ will be
a surface or a surface with boundary.
In the remaining cases,
$n_3=N/2$ and $n_1 \neq n_2$ by
Theorem~\ref{edges}.
By Lemma~\ref{lemma2}, these are the spaces
$C(n_1,n_2,n_1+n_2)$, with $n_1 < n_2$.  By
Theorem~\ref{disj} we may assume that
$\gcd(n_1,n_2)=1$.

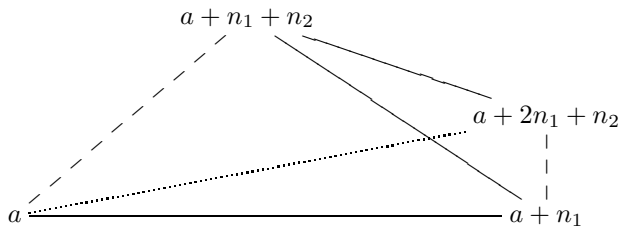
\begin{figure}
\[
\xymatrix{
&&
a+n_1+n_2
\ar@{--}[ddll]
\ar@{-}[drr]
\ar@{-}[ddrr]
& \\
&&&&
a+2n_1+n_2
\ar@{--}[d] 
 \\
a
\ar@{-}[rrrr]
\ar@{.}[rrrru]
&&&&
a+n_1 \\
}
\]
\caption{One tetrahedron boundary in the circle of tetrahedra boundaries.
The tritones are indicated by dashed lines.}
\label{circleoft}
\end{figure}

\begin{theorem}
\label{circleoftetra}
\edit{The simplicial complex} $C(n_1,n_2,n_1+n_2)$ with $n_1<n_2$ is a
circle of $n_1+n_2$ tetrahedra boundaries, each joined to the ones on either side
of it along opposite edges.

\end{theorem}

\begin{proof}
In $C(n_1,n_2,n_1+n_2)$, we see that each vertex $a$ lies in a unique
interval of length $N/2$ $\{a,a+n_1+n_2\}$.
By Theorem~\ref{edges}, this edge is contained in exactly
four 2-simplices, two of which are in each of the following
tetrahedra boundaries.

\begin{enumerate}
\item $   \{a,a+n_1,a+n_1+n_2,a+2n_1+n_2\} $ 
\item  $ \{a,a+n_2,a+n_1+n_2,a+n_1+2n_2\}$
\end{enumerate}
All four of the 2-simplices in each of these tetrahedra boundaries
are in $C(n_1,n_2,n_1+n_2)$. 
The intersection
of these two tetrahedra boundaries is exactly the edge corresponding to the tritone
$\{a,a+n_1+n_2\}$. 
The tetrahedra boundaries can thus be envisioned as linking
tritones in a chain in which we add $n_1$ in one direction and add $n_2$
in the reverse direction (see Figure~\ref{chain}).

This chain forms a circle because each tritone lies in only two tetrahedra boundaries,
and since $\gcd(n_1,n_2,n_1+n_2) = 1$, the space must be connected. 
(Alternatively, there are $N/2$ tritones which we can index by the integers
modulo $N/2 =n_1 + n_2$.
Since $n_1$ is relatively
prime to $N/2$, repeatedly adding $n_1$ must exhaust the tritones.)

  \end{proof}


\begin{figure}
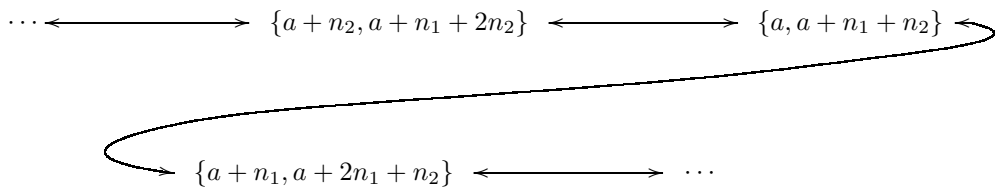

\[
\xy
(0,20)*{\cdots};(50,20)*{\{a+n_2,a+n_1+2n_2\}};
(110,20)*{\{a,a+n_1+n_2\}};(40,0)*{\{a+n_1,a+2n_1+n_2\}};
(90,0)*{\cdots};
(20,0)*{};(126,20)*{};
**\crv{(135,18)&(110,15)&(95,13)&(70,11)&(40,9)&(18,7)&(5,2)}
?>*\dir{>}
\ar@{->}   (126,20);(124,20)
\ar@{<->} (3,20);(30,20)
\ar@{<->} (70,20);(95,20)
\ar@{<->} (60,0);(85,0)
\endxy
\]
\caption{Relation between tritones in a circle of tetrahedra boundaries:
to move right, add $n_1$ to each vertex in the tritone, and
to move left, add $n_2$ to each vertex in the tritone.
}
\label{chain}
\end{figure}

As noted earlier, the circle of tetrahedra boundaries is a quotient of
Peck's Klein bottle \cite{Peck}. In fact, 
the Klein bottle
is a desingularization
of the circle of tetrahedra boundaries.
Precisely,
one obtains the circle of tetrahedra boundaries
by identifying the two distinct copies of each tritone which occur in
Peck's Klein bottles.
The `twist' in the space vanishes when this is done.
\section{Summary and Conclusion}
We have completely classified $C(n_1,n_2,n_3)$ in any case.

\begin{theorem}
\label{sumthm}
The connected components of the generalized {\it Tonnetz} $C(n_1,n_2,n_3)$ are isomorphic and are 2-simplices, tetrahedra boundaries, tori,
cylinders, M\"obius bands, or circles of tetrahedra boundaries. See Figure~\ref{possibilities} for the precise cases. 
\end{theorem}

\edit{
The generic case is $n_1 < n_2 < n_3$ with $n_3 \neq N/2$, and this
is shown in Theorem~\ref{tori} to give a disjoint union of tori.
When $n_3 = N/2$, we must have $n_3 = n_1+n_2$ by Lemma~\ref{lemma2}.
In the most generic of these cases, $n_1 \neq n_2$, the torus
is `pinched' along edges corresponding to tritones to form a `circle
of tetrahedra boundaries' (Theorem~\ref{circleoftetra}).
These are exactly the cases considered by Peck \cite{Peck}
in his Klein bottle {\em Tonnetz}, which are desingularizations of
ours.  
}

\edit{
In the remaining cases, either $n_1 = n_2 < n_3$, $n_1 < n_2 = n_3$,
or  $n_1 = n_2 = n_3$.  
When $n_1 = n_2$, we have the possibility that $n_3 = N/2$, in which case
we have a disjoint union of individual tetrahedra boundaries
(i.e., 2-spheres) (Theorem~\ref{2-sphere}).  (This can be considered a
limiting case of the circle of tetrahedra boundaries, in which the circle 
has only one tetrahedron boundary.)
More generically, $n_3 \neq N/2$ and we obtain a generalization
of Mazzola's harmonic strip, which we identify as either a
M\"obius band or cylinder, according to the parity of the number of
notes in the octave (Theorem~\ref{nnnplusk}).
When $n_1 < n_2 = n_3$, the possibility that $n_3 = N/2$ no longer exists
and we again have harmonic strips (Theorem~\ref{nnnplusk}).
Finally, if $n_1 = n_2 = n_3$, we are discussing augmented triads
$\{C, E, G^{\sharp}\}$
and their microtonal transpositions (Theorem~\ref{triangle}).
}

\edit{
By re-envisioning the {\em Tonnetz} as a simplicial complex, rather than as
a quotient of a planar region, we have sharpened the focus on the relationships
between the chords in the {\em Tonnetz} by using exactly (and only) the relations
between the intervals in the chords to assemble the space.  
This abstraction allows us to define
generalized {\em Tonnetze} for triads of arbitrary shape in scales 
which divide the octave into any number of steps.
}

\edit{
When the chord has a `common denominator', i.e., in the cases
$C(dn_1,dn_2,dn_3)$,
the resulting {\em Tonnetz} is a disjoint union of identical subspaces
which are generalized {\em Tonnetze} 
$C(n_1,n_2,n_3)$
in their own right. The $T/I$ group then
cycles through these components, while the $P$, $L$, and $R$ operations are
confined to each component.  
}

\edit{
Interestingly, all possibilities except the M\"obius band occur  in
the chromatic scale $\Z/12$, as noted in the Introduction.  (The 
M\"obius band can only occur in scales which divide the octave 
into an odd number of steps.)
}

\section*{Acknowledgement(s)}
I would like to thank Sean Gavin for all his helpful ideas and advice, as well as 
Tom Fiore for his discussions and expertise in both mathematics and music theory.  I
want to thank Robert Bruner for all his guidance and help with this paper.
Without his time and patience, this work would not have been
possible. I would also like to thank a very diligent referee, whose numerous
and detailed suggestions greatly improved the content and readability of this paper.
This work was funded in part by the Office of Undergraduate 
Research of Wayne State University. It was also funded \edit{by}
NSF grant PHY-0851678.

\end{document}